\documentclass[11pt]{article}

\usepackage[utf8]{inputenc}
\usepackage{amssymb} 
\usepackage{amsmath} 
\usepackage{mathrsfs} 
\usepackage{amsthm} 
\usepackage{amsfonts} 
\usepackage{thmtools} 
\usepackage{mathtools}
\usepackage{enumitem}
\usepackage{graphicx,psfrag}
\usepackage{sidecap} 
\usepackage[T1]{fontenc}
\usepackage{caption}
\usepackage{subcaption}
\usepackage{fix-cm}
\usepackage{titling} 
\usepackage{sectsty} 
\usepackage{dsfont}
\usepackage{nicefrac}
\usepackage[svgnames]{xcolor}
\usepackage{cleveref}
\usepackage{stackengine}
\usepackage{scalerel}
\usepackage{accents}
\usepackage{subdepth}
\usepackage[numbers]{natbib}

\usepackage[a4paper,left=2.6cm,right=2.6cm,top=3.9cm, bottom=3.3cm, bindingoffset=5mm, headsep=8mm,headheight=1.7cm]{geometry}

\usepackage{fancyhdr} 
\providecommand{\keywords}[1]{\usefont{OT1}{phv}{b}{n}  Keywords: \normalfont #1}
\providecommand{\subclass}[1]{MSC 2010 Subject classification: \normalfont #1}

\newcommand{\HorRule}{\rule{\linewidth}{.2pt}} 

\pretitle{\vspace{-30pt} \begin{flushleft} \fontsize{20}{20} \usefont{OT1}{phv}{b}{n} \selectfont} 

\title{Article Title}

\posttitle{\par \end{flushleft}\vskip 0.5em} 

\preauthor{\begin{flushleft}\large \lineskip 0.5em \usefont{OT1}{phv}{m}{n} } 

\author{John Smith, }

\postauthor{ \\[1em] \fontsize{8}{8} \usefont{OT1}{phv}{m}{n}  Faculty of Technology, Bielefeld University, Box 100131, 33501 Bielefeld, Germany
\\[.4em]
\normalsize \usefont{OT1}{phv}{m}{n}

\par \vskip -.5em \end{flushleft} \HorRule} 

\date{}

\def\abstract{\topsep=0pt\partopsep=0pt\parsep=0pt\itemsep=0pt\relax
\trivlist\item [\hskip\labelsep
{\usefont{OT1}{phv}{b}{n} \abstractname}]\if!\abstractname!\hskip-\labelsep\fi }

\allsectionsfont{\large \usefont{OT1}{phv}{b}{n}} 
\paragraphfont{\normalsize }
\subsectionfont{\normalsize \usefont{OT1}{phv}{b}{n}}
\makeatletter
\let\origsection\section
\renewcommand\section{\@ifstar{\starsection}{\nostarsection}}

\newcommand\nostarsection[1]
{\sectionprelude\origsection{#1}\sectionpostlude}

\newcommand\starsection[1]
{\sectionprelude\origsection*{#1}\sectionpostlude}

\newcommand\sectionprelude{%
  \vspace{1em}
}

\newcommand\sectionpostlude{%
  \vspace{-.7em}
}
\makeatother

\declaretheoremstyle[
    spaceabove=6pt, 
    spacebelow=6pt, 
    headfont=\sffamily\bfseries,
    notefont=\normalfont, 
    notebraces={(}{)}, 
    bodyfont=\normalfont,
    postheadspace=1em,
    qed=$square$,
    headpunct={.}]{myproof}

\declaretheoremstyle[
    spaceabove=6pt, 
    spacebelow=6pt, 
    headfont=\sffamily\bfseries,
    notefont=\normalfont, 
    notebraces={(}{)}, 
    bodyfont=\normalfont,
    postheadspace=1em,
    qed=$\diamondsuit$,
    headpunct={.}]{mystyle}
\declaretheorem[name={Example}, style=mystyle]{example}
\declaretheorem[name={Remark}, style=mystyle]{remark}

\declaretheoremstyle[
    spaceabove=6pt, 
    spacebelow=6pt, 
    headfont=\sffamily\bfseries,
    notefont=\normalfont, 
    notebraces={(}{)}, 
    bodyfont=\normalfont,
    postheadspace=1em,
    headpunct={.}]{mystyle}

\declaretheoremstyle[
    spaceabove=6pt, 
    spacebelow=6pt, 
    headfont=\sffamily\bfseries,
    notefont=\normalfont, 
    notebraces={(}{)}, 
    bodyfont=\normalfont,
    postheadspace=1em,
    headpunct={.}]{mydef} 
\declaretheorem[name={Definition}, style=mydef]{definition}
\declaretheorem[name={Fact}, style=mystyle]{fact}

\declaretheoremstyle[
    spaceabove=6pt, 
    spacebelow=6pt, 
    headfont=\sffamily\bfseries,
    notefont=\normalfont, 
    notebraces={(}{)}, 
    bodyfont=\normalfont\itshape,
    postheadspace=1em,
    headpunct={.}]{mythm} 
\declaretheorem[name={Theorem}, style=mythm]{theorem}
\declaretheorem[name={Lemma}, style=mythm]{lemma}

\declaretheorem[name={Proposition}, style=mythm]{proposition}

\declaretheoremstyle[
    spaceabove=6pt, 
    spacebelow=6pt, 
    headfont=\sffamily\bfseries,
    notefont=\normalfont, 
    notebraces={(}{)}, 
    bodyfont=\normalfont\itshape,
    postheadspace=1em,
    qed=$square$,
    headpunct={.}]{mycor} 
\declaretheorem[name={Corollary}, style=mythm]{corol}

\makeatother

\SetLabelAlign{left}{#1\hfil}

\newenvironment{tightitemize}
{ \begin{itemize}
    \setlength{\itemsep}{3pt}
    \setlength{\parskip}{0pt}
    \setlength{\parsep}{0pt}    }
{ \end{itemize}                  } 

\newcommand{\BIGOP}[1]
{\mathop{\mathchoice%
{\raise-0.22em\hbox{\huge $#1$}}%
{\raise-0.05em\hbox{\Large $#1$}}{\hbox{\large $#1$}}{#1}}}
\newcommand{\bigtimes}{\BIGOP{\times}}
\let\oldcap\cap
\def\cap{\,\oldcap\,}
\renewcommand{\P}{\mathbf P} 
\newcommand{\NN}{\mathbb N_0}
\newcommand{\ts}{\hspace{0.8pt}}
\newcommand{\tts}{\hspace{1pt}}
\newcommand{\nts}{\hspace{-0.5pt}}

\renewcommand{\L}[2]{\mathcal L^{#2}_{#1}}
\newcommand{\Pot}[1]{\mathcal P(#1)}
\newcommand{\F}[2]{F^{#2}_{#1}}
\newcommand{\wF}[2]{\widehat F^{#2}_{#1}}
\newcommand{\Af}[2]{A^{\scriptstyle \nts #2}_{#1}}
\newcommand{\rh}[1]{r^{}_{#1}}
\newcommand{\rc}[1]{\varrho^{}_{#1}}
\newcommand{\lam}[2]{\lambda_{#1}^{#2}}

\renewcommand{\t}[1]{\mathcal T^{}_{#1}}

\newcommand{\Tp}[2]{T_{#1}(#2)}
\newcommand{\Tps}[2]{T_{#1}^L(#2)}
\newcommand{\Ts}{T_{}^{\scriptscriptstyle L}}
\newcommand{\I}[1]{I_{#1}^{}}
\newcommand{\Cc}[1]{\mathcal C(#1)}
\newcommand{\G}[2]{G^{}_{#1} (#2)}
\newcommand{\Rr}[1]{\mathcal R(#1)}
\renewcommand{\S}{\mathcal S}
\newcommand{\Ileq}[1]{I_{#1}^{\ts l}}
\newcommand{\Igeq}[1]{I_{#1}^{\ts r}}

\newcommand{\<}{\preccurlyeq}
\renewcommand{\>}{\succcurlyeq}
\newcommand{\pF}{\preceq_{\ts \scriptscriptstyle F}}
\newcommand{\pD}{\preceq_{\ts \scriptscriptstyle P}}

\newcommand{\PP}{\mathbb P}

\newcommand{\Ftree}{F_t^{}\leftrightarrow \Ts}

\makeatletter
\def\m#1{\def\tempa{#1}\futurelet\next\m@i}
\def\m@i{\ifx\next\bgroup\expandafter\m@ii\else\expandafter\m@end\fi}
\def\m@ii#1{m^{}_{\tempa}(#1)}
\def\m@end{m^{}_{\tempa}}
\makeatother

\makeatletter
\def\mh#1{\def\tempa{#1}\futurelet\next\mh@i}
\def\mh@i{\ifx\next\bgroup\expandafter\mh@ii\else\expandafter\mh@end\fi}
\def\mh@ii#1{\stackon[-4.3pt]{\text{$m$}}{\vstretch{1.15}{\hstretch{2.5}{\hat{\phantom{\;}}}}}^{}_{\tempa}(#1)}
\def\mh@end{\stackon[-4.3pt]{\text{$m$}}{\vstretch{1.15}{\hstretch{2.5}{\hat{\phantom{\;}}}}}^{}_{\tempa}}
\makeatother

\DeclareMathOperator{\Maxx}{MAX}
\DeclareMathOperator{\Minn}{MIN}
 
\makeatletter
\def\Max#1{\def\tempa{#1}\futurelet\next\Max@i}
\def\Max@i{\ifx\next\bgroup\expandafter\Max@ii\else\expandafter\Max@end\fi}
\def\Max@ii#1{\Maxx_{\ts t,\tts \tempa}(#1)}
\def\Max@end{\Maxx_{\ts t,\tts \tempa}^{}}
\makeatother

\makeatletter
\def\Min#1{\def\tempa{#1}\futurelet\next\Min@i}
\def\Min@i{\ifx\next\bgroup\expandafter\Min@ii\else\expandafter\Min@end\fi}
\def\Min@ii#1{\Minn_{t,\tts \tempa}(#1)}
\def\Min@end{\Minn_{t,\tts \tempa}^{}}
\makeatother

\newcommand{\X}[2]{X_{#1}^{#2}}
\newcommand{\w}[2]{\omega_{#1}^{#2}}
\newcommand{\wind}[1]{\omega_{\text{ind}}^{#1}}
\newcommand{\wdep}[1]{\omega_{\text{dep}}^{#1}}
\newcommand{\Omeg}[1]{\Omega_{}^{#1}} 

\newcommand{\Ev}{B}
\newcommand{\Ef}{B_{\nts \scriptscriptstyle F}}
\newcommand{\Ex}{B_{\nts \scriptscriptstyle X}}

\newcommand*\xbar[1]{%
  \hbox{%
   \vbox{%
      \hrule height 0.25pt 
      \kern0.6ex
      \hbox{%
        \ensuremath{#1}%
         }}}} 
         
\newcommand*\tbar[1]{%
\nts\!\nts \hbox{
   \vbox{%
    \hrule height 0.3pt 
      \kern0.4ex
      \hbox{%
        \kern -0.01 em
        \ensuremath{#1}%
         }}}}

\usepackage{tikz}
\usetikzlibrary{decorations.text}
\usetikzlibrary{shapes}
\usetikzlibrary{shapes.callouts,positioning}
\usetikzlibrary{arrows,positioning, automata} 
\usepackage{pst-pdf}
\usepackage{pst-math} 
\usetikzlibrary{decorations.markings}

\begin{document}

\title{Fragmentation process, pruning poset for rooted forests, and M\"obius inversion}

\author{Ellen Baake \thanks{E-mail: ebaake@techfak.uni-bielefeld.de} \;\,-- \,Mareike Esser \thanks{E-mail: messer@techfak.uni-bielefeld.de} }
\date{}

\maketitle

\thispagestyle{fancy}

\begin{small}
\abstract{We consider a discrete-time Markov chain, called fragmentation process, that describes a specific way of successively removing objects from a linear arrangement. The process arises in population genetics and describes the ancestry of the genetic material of individuals in a population experiencing recombination. We aim at the law of the process over time. To this end, we investigate sets of realisations of this process that agree with respect to a specific order of events and represent each such set by a rooted (binary) tree. 
The probability of each tree is, in turn, obtained by Möbius inversion on a suitable poset of all rooted forests that can be obtained from the tree by edge deletion; we call this poset the \textit{pruning poset}. Dependencies within the fragments make it difficult to obtain explicit expressions for the probabilities of the trees. We therefore construct an auxiliary process for every given tree, which is i.i.d. over time, and which allows to give a pathwise construction of realisations that match the tree. }


\subclass{05A15 (primary),\,60J20 (secondary),\,05C05,\,06B99}

\end{small}


\section{Introduction}\label{sec:introduction}

Consider a linear arrangement of $n$ discrete objects captured in the set $L=\{1,\ldots,n\}$. We like to think of $L$ as a chain, and the elements of $L$ as \textit{links} of the chain, in the sense of the connecting components of a real-world chain. If links are removed, the remaining set of links splits into contiguous fragments. We will investigate the probability distribution of a Markov chain $(\F{t}{})^{}_{t\in \NN}$, $\NN:=\{0,1,2,\ldots\}$, on the set of subsets of $L$, where $\F{t}{}$ is the set of links removed until  time $t$. The details of the process will be described below; let us only note here that $\F{0}{}=\varnothing$ and that, at every point in time, \textit{at most} one link is removed from \emph{every} fragment with a given probability or rate. We focus on the discrete version of the process and briefly mention simplifications arising for the continuous-time analogue in Section~\ref{sec:tree_prob_cont}. We speak of $(\F{t}{})^{}_{t\in \NN}$ as the \emph{fragmentation process}. 

The fragmentation process and its probability distribution originate from the context of population genetics, or more precisely from the evolution of a sufficiently large population under recombination. \emph{Recombination} occurs during sexual reproduction and refers to the reshuffling of the genetic material of two parents into a `mixed'-type offspring individual, see Figure~\ref{fig:reco_and_ancestry} (left). As illustrated in Figure~\ref{fig:reco_and_ancestry} (right), in the backward perspective, each crossover causes a split of the genetic material into two contiguous fragments; one part is inherited from the mother, one from the father. If we identity the splitting events with the removal of links and assume that there is at most one splitting per pair of gene sequences (which is mostly true even for fairly large genomic regions \cite{Hillers}), the fragmentation process describes how the genetic material of a single individual is distributed across its ancestors. For more on this, see \cite{BaakeBaake2, BaakeWangenheim, Bennett, Dawson, EsserProbstBaake,Geiringer, Jennings,Martinez,Robbins,WangenheimBaakeBaake}.

\begin{figure}[bp]
\includegraphics[width=.4\textwidth]{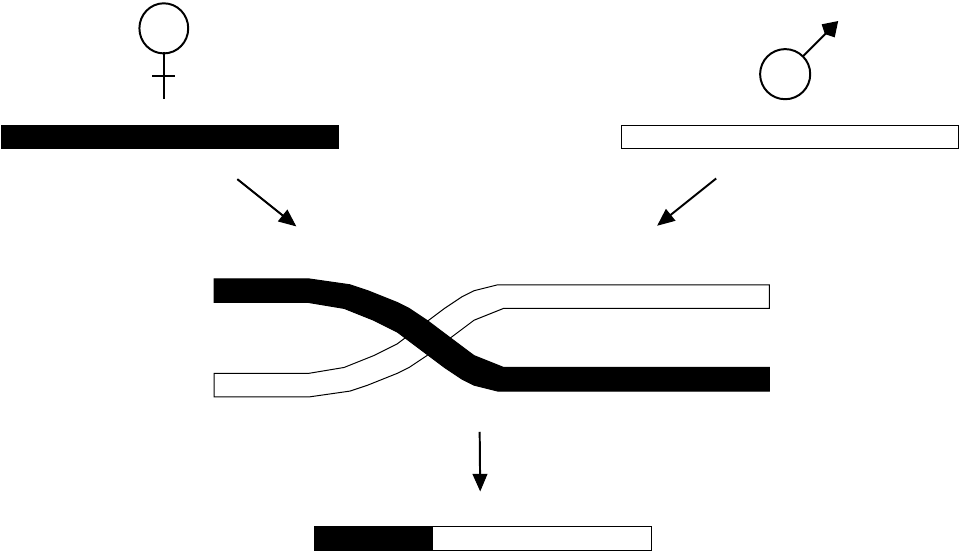}
\hfill
\includegraphics[width=.4\textwidth]{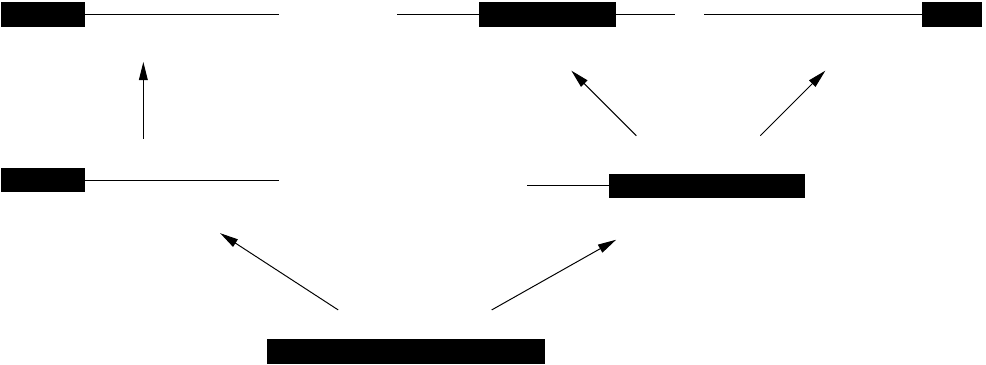}
\caption{\label{fig:reco_and_ancestry}Left: Simplified representation of a recombination event in which maternal and paternal gene sequences cross over and physically swap their genetic material to form a mixed-type offspring. Right: Ancestry of a single individual backward in time (here we go back for two generations in discrete time). The genetic material of the individual is dispersed across the two parents. Restriction to at most one crossovers ensures that the sequence splits into two contiguous blocks. Horizontal lines indicate nonancestral material, i.e. genetic material that is not passed on to the offspring and thus irrelevant for it.}
\end{figure}

In \cite{BaakeWangenheim}, we have worked out the (marginal) distribution of $\F{t}{}$ for every given $t$. Since removing a given link forbids to remove any other link in the same fragment in the same time step, links are dependent as long as they belong to the same fragment, and become independent once they have been separated. 
For each realisation of $(\F{t}{})^{}_{t\in \NN}$, the order of events therefore matters. One may thus collect all realisations of $(\F{t}{})^{}_{t\in \NN}$ that agree on the order of events and that end up in a particular state $G$ at time $t$ and represent this set of realisations as a rooted binary tree. Here, the elements of $\F{t}{}$ are identified with the vertices of that tree in such a way that the relevant time order of events of the fragmentation process is encoded by the partial order on the vertices of the tree. In \cite{BaakeWangenheim}, the probability that $\F{t}{}=G$, $G\subseteq L$, for some $t\in \NN$, is given as a sum over all probabilities related to trees with vertex set $G$. The probability for each individual tree was obtained from a technical calculation by summing over all possible combinations of branch lengths, i.e. over all possible combinations of times that $\F{t}{}$ spends in the various states. This summation led to an alternating sum over terms that reflects a decomposition of the tree into subtrees. The result provided an answer to the problem, but was somewhat unsatisfactory since both the combinatorial and the probabilistic meanings remained in the dark. 
As to the combinatorial side, the alternating sum hinted at an underlying, yet unidentified, inclusion-exclusion principle; an instance of the "wide gap between the bare statement of the principle and the skill required in recognizing that it applies to a particular combinatorial problem", as stated by Rota in his seminal work on M\"obius functions \cite{Rota}. As to the probabilistic side, the terms in the sum hinted at some underlying independence across subtrees, but were hard to interpret in detail. 

The purpose of this article is to provide both the combinatorial and the probabilistic insight into the solution of the fragmentation process, and thus to give a conceptual proof for it. 
On the combinatorial side, we will introduce a suitable poset (to be called the \emph{pruning poset}) on rooted forests that can be obtained from a rooted tree via specific edge deletion similar to the cutting-down procedure introduced by Meir and Moon \cite{MeierMoon}. In \cite{MeierMoon}, the root of a random tree is isolated by uniformly cutting edges of the tree until the tree is reduced to the root. 
In the resulting line of research (see, for example, \cite{BaurBertoin,DrmotaIksanovMoehleRoesler,Janson,MeierMoon,Panholzer}), one is interested in the distribution and limiting behaviour of the number of cuts required to isolate the root for various classes of random trees. In contrast, in our pruning, we keep track of the entire rooted forest, rather than the single subtree that contains the root. 
It turns out that the pruning poset is a special case of the \textit{poset of planar forests} introduced by Foissy \cite{Foissy}. Fortunately, our special case has enough structure to allow for a simple and explicit Möbius function for trees of arbitrary size; this is, so far, unavailable in Foissy's larger poset. The explicit Möbius inversion formula will lead to the inclusion-exclusion principle underlying our tree probabilities. \\

On the probabilistic side, we construct an \emph{auxiliary process} with \emph{time-independent law}, from which one can read off the (marginal) law of  $(\F{t}{})^{}_{t\in \NN}$, for every given $t$, via a pathwise construction. The method is reminiscent of that used by Clifford and Sudbury \cite{CliffordSudbury}, who jointly represent \textit{all } transitions of a given Markov chain on the same probability space, provided the process is monotone and has a totally ordered state space. The state space of $(\F{t}{})^{}_{t\in \NN}$, however, is only partially ordered; we therefore need a version that works on a tree-by-tree basis. This allows one to cope with the dependence of the links and the resulting state dependence of the transitions of $(\F{t}{})^{}_{t\in \NN}$. 

The article is organised as follows. In Section~\ref{sec:fragmentation_process}, we first define the fragmentation process $(\F{t}{})^{}_{t\in \NN}$ and relate sets of realisations of $(\F{t}{})^{}_{t\in \NN}$ to certain rooted trees. We then (Section~\ref{sec:moebius_pruning_poset}) construct a pruning poset for general rooted forests, find its Möbius function, and give the corresponding Möbius inversion principle. We then use Möbius inversion on the pruning poset to obtain an expression for the tree probabilities (Section~\ref{sec:moebius_fragmentation_trees}). The final explicit expression for the tree probabilities (Section~\ref{sec:expl_tree_prob_discrete}) will rely on the auxiliary process, which is introduced in Section~\ref{sec:auxiliary_process}.


\section{Fragmentation process and fragmentation trees}\label{sec:fragmentation_process}

The state space of the fragmentation process $(\F{t}{})^{}_{t\in \NN}$ is $\Pot{L}$, where $L=\{1,\ldots,n\}$ is the set of discrete objects called links and $\F{t}{}$ denotes the set of links that have been removed until time $t$. The initial state of $(\F{t}{})^{}_{t\in \NN}$ is $\F{0}{}=\varnothing$, the absorbing state is $L$, and $\F{t'}{}\subseteq \F{t}{}$ for all $t'<t$. 
If a link $\alpha\in L$ is removed, the remaining set of links is decomposed into the contiguous fragments $\{\beta\in L: \beta<\alpha\}$ and $\{\beta\in L: \beta>\alpha\}$. If all links in $G=\{\alpha^{}_1,\ldots,\alpha^{}_{\vert G\vert}\}\subseteq L$ with
$\alpha^{}_1<\alpha^{}_2<\cdots <\alpha^{}_{\vert G\vert}$ are removed, 
$G$ induces a decomposition of the set of remaining fragments of links into
\begin{equation} \label{def:L_G}
\L{G}{}:=\big\{J^{}_1,\ldots,J^{}_{|G|+1} \big\},
\end{equation}
where
\[
J_1^{}=\{ \alpha \in L :  \alpha < \alpha_1 \}, \ \ J_2^{}=\{ \alpha \in L :   \alpha_1 < \alpha < \alpha_2 \},\ \ \ldots, \ \  J^{}_{|G|+1}=\{ \alpha \in L :   \alpha^{}_{|G|} < \alpha\};
\]
in particular, $\L{\varnothing}{} =\{L\}$ and $\L{L}{} =\{\varnothing\}$. Clearly, the $J_i$ may be empty and $\L{G}{}\setminus \{\varnothing\}$ is a partition of $L \setminus G$.

\begin{definition}[Fragmentation process] \label{def:Fragmentation_discrete}
Let $\rh{\alpha}$, $\alpha \in L$, be positive with $\sum_{\alpha \in L} \rh{\alpha}\leqslant 1$. $(\F{t}{})_{t \in \NN}^{}$ is the following discrete-time Markov chain with values in $\Pot{L}$: The initial state is $\F{0}{} = \varnothing$ and, conditional on $\F{t-1}{}=G$,
\[
\F{t}{} = \F{t-1}{} \cup \Big(\bigcup_{J \in \L{G}{}}\Af{t}{J} \Big).
\]
Here $\Af{t}{J} = \{ \alpha \}$ with probability $\rh{\alpha}$ for all $\alpha\in J$ and $\Af{t}{J} = \varnothing$ with probability                                                            $1-\sum_{\alpha\in J} \rh{\alpha}$, independently for all $J\in \L{G}{}$ and all $t\geqslant 1$, and $\L{G}{}$ is defined as in \eqref{def:L_G}.
\end{definition}

The definition deals consistently with empty fragments since $\Af{t}{\varnothing} = \varnothing$ with probability 1. 

\begin{remark}
Let us mention for completeness that Definition~\ref{def:Fragmentation_discrete} immediately entails the transition probabilities
\[
 \PP(F_t=H \mid F_{t-1}=G) = \begin{cases}  \prod_{J \in \L{G}{}} p^{}_J, & \lvert (H \setminus G) \cap J \rvert \leqslant 1 \text{ for all } J \in \L{G}{} , \\ 0, & \text{otherwise}, \end{cases}
\]
where
\[
 p^{}_J := \begin{cases} \rh{\alpha}, & (H \setminus G) \cap J = \{\alpha\} \text{ for some } \alpha \in J, \\
 1-\sum_{\alpha \in J} \rh{\alpha}, & (H \setminus G) \cap J = \varnothing. \end{cases}
\]
In what follows, however, this explicit representation will not be required.
\end{remark}

Links are dependent as long as they belong to the same fragment and become independent once they are separated on different fragments. We can therefore represent $(\F{t}{})^{}_{t\geqslant t'}$  as
\begin{equation}\label{eq:recursion_F}
\F{t}{L}=\F{t'}{L}\cup \bigg(\bigcup_{J\in \L{F_{t'}^{L}}{}} \F{t}{J}\bigg),\quad t'\geqslant 0,\ t\geqslant t'.
\end{equation}
The $(\F{t}{J})^{}_{t\geqslant t'}$'s are independent processes with $\F{t'}{J}=\varnothing$ and $(\F{t}{J}\big)^{}_{t\geqslant t'}$ defined in analogy with $(\F{t}{L})_{t\in \NN}^{}:= (\F{t}{})_{t\in \NN}^{} $. That is, $(\F{t}{J})^{}_{t\geqslant t'}$ is the fragmentation process defined on the underlying set of links $J$ with removal probabilities $\rh{\ts \alpha}{}$, $\alpha \in J$. Throughout, we use the upper index to indicate the underlying set of links and may omit it if the set is $L$.


Our interest is in the (marginal) probability distribution of $\F{t}{}$ for any given $t$. We will throughout rely on a formulation via waiting times. Let $\t{\alpha}:= \min\{t\in \NN:\alpha \in \F{t}{}\}$ be the waiting time for  link $\alpha$ to be removed and $\t{K}:= \min\{\t{\alpha}: \alpha \in K\}$ the time at which the first link in $K \subseteq L$ is removed.
The event $\{\F{t}{}=G\}$  then obviously translates into
\begin{equation}\label{eq:distribution_Ft}
\{\F{t}{}=G\} =\big\{\max\{\t{\alpha}:\alpha\in G\}\leqslant t < \t{ L\setminus G}\}\,,
\end{equation}
for every $G\subseteq L$ and $t \geqslant 0$.

Since links are not independent in general and dependencies change over time, $\P\big(\F{t}{}=G\big)$ fails to have an obvious explicit expression. As mentioned in the Introduction, interest therefore shifts to classes of realisations of the fragmentation process that  end up in the state $G$ at time $t$ and that agree on the time order of some events. Each such set will be represented by an augmented version of a rooted binary tree. This will be done next. 


Let $T=(\gamma, G, E)$ denote a plane oriented binary tree with root $\gamma$, vertex (or node) set $G=G(T)$, set of edges $E=E(T)\subseteq G\times G$ and standard partial order $\<$ on the set of vertices. 
If $\alpha$ and $\beta$ are adjacent and $\alpha\prec\beta$, we write $e=(\alpha,\beta)$ and call $\alpha$ and $\beta$ the \textit{ends} of $e$; more precisely, $\alpha$ is the \textit{lower end} and $\beta$ the \textit{upper end} of $e$. The partial order on $G$ obviously induces a partial order on $E$ (via the partial order of the upper ends, say), which we will (by slight abuse of notation) also denote by $\<$. The set of all (order) ideals of $(G,\<)$ (see for instance \cite[p.~100]{Stanley1}) will be denoted by $\Rr{T}$, that is, 
\begin{equation}\label{def:R(T)}
\Rr{T}:=\{ R\subseteq G\, \mid\, R\neq \varnothing \text{ and for every }\alpha\in R,\ \beta\<\alpha \text{ implies } \beta\in R\}.
\end{equation}

Let $\S:=\bigcup_{R\in\Rr{T}\,\cup\,\varnothing} \L{R}{}$ be the set of all (possibly empty) fragments that emerge when links are removed from $L$ in the order prescribed by $T$, where $\L{R}{}$ is defined as in \eqref{def:L_G}. Cleary, $\S$ depends on $T$, but we suppress the dependence on $T$ in the notation. A \textit{fragmentation tree} $\Ts := (\gamma, G, E, L)$ corresponding to the tree $T=(\gamma,G,E)$ is then the augmented planted plane tree constructed as follows (see Fig.~\ref{fig:fragmentation_tree} for an example): 
\begin{tightitemize}
\item Add additional lines to $T$ such that every vertex $\alpha \in G$ has exactly two lines emanating from it. We call these additional lines \textit{branches} and distinguish them from edges. More precisely, a \textit{branch} has a lower end and no upper end in the vertex set of $T$, whereas an edge always connects two vertices. 
\item Add a \textit{phantom} node $r$ to the tree. That is, $r$ is the parent of $\gamma$, but does not count as a vertex (this makes $\Ts$ a planted plane tree \cite{Drmota}). Connect $r$ and $\gamma$ by a branch.
\item Associate every line (edge or branch) with a fragment $J\in\S$ according to the following rules. Start with the line between $r$ and $\gamma$ and identify it with $\I{\gamma}=L$. Next, associate the two lines emanating from $\gamma$ with the fragment $\Ileq{\gamma}:=\{\beta\in \I{\gamma}:\beta<\gamma\}$ and $\Igeq{\gamma}:=\{\beta\in \I{\gamma}:\beta>\gamma\}$; so $\Ileq{\gamma}$ is the left and $\Igeq{\gamma}$ the right branch or edge, and $\I{\gamma}=\Ileq{\gamma}\cup\{\gamma\} \cup \Igeq{\gamma}$ as well as $\L{\{\gamma\}}{}=\{\Ileq{\gamma},\Igeq{\gamma}\}$. If $\gamma$ has a child $\alpha\in G$ ($\beta\in G$) with $\alpha<\gamma$ ($\gamma<\beta$), we set $\Ileq{\gamma}=:\I{\alpha}$  ($\Igeq{\gamma}=:\I{\beta}$) and proceed up the tree in a recursive way by identifying all remaining lines with the (possibly empty) fragments $J\in\S$ in an analogous way, starting with the lines emanating from the child(ren) of $\gamma$.
\end{tightitemize}

\begin{figure}
\setlength{\fboxsep}{4pt}
\includegraphics[width=.85\textwidth]{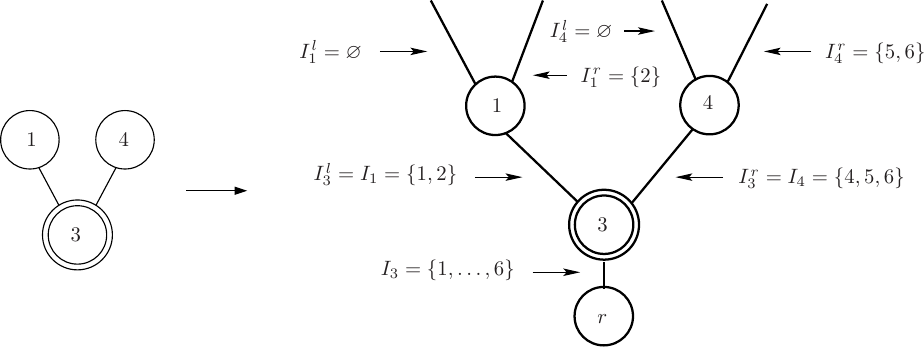}
\caption{\label{fig:fragmentation_tree}
Left: Plane oriented tree $T=(\gamma,G,E)$ with vertex set $G= \{1,3,4\}$ and root $\gamma=3$. Right: The corresponding fragmentation tree $\Ts=(\gamma,G,E,L)$ for $L=\{1,\ldots,6\}$. Here, $\S{}=\{\I{3},\I{1},\I{4},\varnothing,\Igeq{1},\Igeq{4}\}$.}
\end{figure}

For every $\alpha\in G$, the fragment $\I{\alpha}$ is the smallest fragment in $\S$ that contains $\alpha$, i.e. the specific fragment from which link $\alpha\in G$ is removed.  $\I{\alpha}$ will be understood as \textit{internal fragment}. The fragments in $\L{G}{}$, namely those that are associated with branches rather than edges, will be termed \textit{external fragments}. External fragments $J\in\L{G}{}$ can be either \textit{full} (if $J\neq \varnothing$) or \textit{empty} (if $J=\varnothing$). For $G=\varnothing$, the only fragmentation tree is the empty planted tree (with no node except the phantom node $r$ and the single line $\I{\gamma}$). 

Due to the above description, we can rewrite $\S$ in various ways, namely,
\[
\S=\bigcup_{R\in\Rr{T}\cup\varnothing} \L{R}{} = \{\I{\gamma}\}\cup\{\Ileq{\alpha}, \Igeq{\alpha}: \alpha \in G\} = \{ I^{}_\alpha : \alpha \in G \} \cup \L{G}{}\,.
\]
In a similar manner, we can write $\L{R}{}$,  $R\in\Rr{T}$, as a collection of external and internal fragments and  internal fragments, namely
\begin{equation}\label{eq:props_I_alpha}
\L{R}{}=\Big(\L{G}{}\setminus \Big(\bigcup_{\alpha\in M(G\setminus R)} \L{\I{\alpha}\cap G}{\I{\alpha}}\Big)\Big) \cup \big\{ \I{\alpha}:\alpha \in M(G\setminus R)\big \},
\end{equation}
where  $\L{R}{\I{\alpha}}$ is defined as in \eqref{def:L_G}  with $L$ replaced by $\I{\alpha}$, $\alpha \in G$, and $M(G\setminus R)$ is the set of vertices in $G\setminus R$ that are minimal with respect to $\<$ (with $M(\varnothing):=\varnothing$). 

\begin{remark} \label{remark:tree_notation}
Our fragmentation trees correspond to the \emph{tree topologies} that occurred in \cite{BaakeWangenheim}. In the genealogical context, a tree topology  means an unweighted tree. We slightly adjusted the notation here for compatibility with the general usage in graph theory. 
\end{remark}

We now match realisations of the fragmentation process with fragmentation trees.  Recall that $\t{\alpha}$ is the waiting time until link $\alpha$ is removed and that $\t{K}$, $K\subseteq L$, is the waiting time until the first link in $K$ is removed. 

\begin{definition}
For a given $t\in \NN$, we say that $(\F{t'}{})^{}_{0\leqslant t' \leqslant t}$ \emph{matches} the fragmentation tree $\Ts=(\gamma,G,E,L)$ if $\F{t}{}=G$ and $\t{\alpha} \leqslant \t{\beta}$ precisely for those $\alpha, \beta \in G$ with $\alpha \< \beta$; in words, if the partial order of the waiting times with respect to links on the same path away from the root agrees with the partial order of the vertices of the fragmentation tree. That is,  we do not care about the order of the waiting times for links that are already separated on different fragments. 
\end{definition}

\begin{figure}[hbtp]
\centering
\setlength{\fboxsep}{5pt}
\includegraphics[width=.9\textwidth]{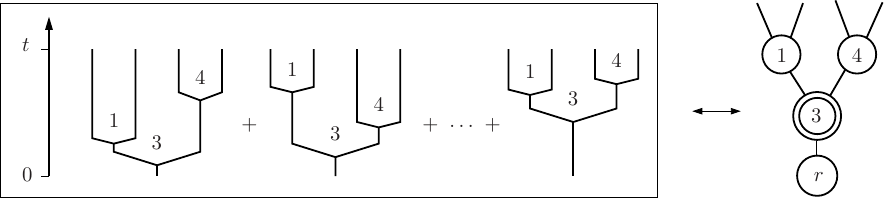}
\caption{\label{fig:fragmentation_process2_on_trees} All realisations of the fragmentation process that end up in the state $\F{t}{}=\{1,3,4\}$ at time $t$ and for which $3$ is removed before $1$ and $4$ match the fragmentation tree on the right.}
\end{figure}
 
Let now $\tau(G,L)$ be the set of all fragmentation trees with vertex set $G$ and underlying link set $L$ (the cardinality of this set is the Catalan number $\textstyle C^{}_{|G|}=\frac{1}{|G|+1} \binom{2|G|}{|G|}$) and define $\G{\alpha}{\varnothing}:=\{\beta\in G \,\mid\, \beta\> \alpha\}$ for every $\alpha\in G$. We can then expand \eqref{eq:distribution_Ft} into
\[
\{\F{t}{}=G\}= \bigcup_{\Ts \in \tau(G,L)} \{\Ftree\}, \quad G\subseteq L,
\]
where
\begin{equation}\label{def:tree_event_tau}
\{\Ftree\}= \big\{\max\{\t{\alpha}:\alpha\in G\}\leqslant t,\, t < \t{L\setminus G},\,\t{\alpha}=\t{I_{\alpha} \cap G} \text{ for all }  \alpha\in G \big\}
\end{equation}
is the event that $(\F{t}{})^{}_{0\leqslant t' \leqslant t}$ \emph{matches} $\Ts$. Indeed,
the \textit{inequalities} in \eqref{def:tree_event_tau} ensure that precisely the vertices in $G$ have been removed before $t$. The \textit{equalities} then enforce the partial order  within the tree  by requiring  that  $\alpha$ be the first link to be removed in the subtree with root $\alpha$ (which has vertex set $\I{\alpha}\cap G$); it is sufficient to look at the links in $\I{\alpha}\cap G$ since we know from the inequalities that those in $\I{\alpha} \setminus G$  are not cut until $t$ anyway. 

The task for the following sections is to find an explicit expression for the (marginal) probabilities $\P(\Ftree)$. To this end, we will first construct a pruning poset on general rooted forests, find its M\"obius function, apply it to our fragmentation trees, and use M\"obius inversion to write the maximum in \eqref{def:tree_event_tau} in terms of minima over certain subsets of $G$.  In Section~\ref{sec:expl_tree_prob_discrete} we will then give an explicit expression for the minima by use of an auxiliary process that is constructed in Section~\ref{sec:auxiliary_process}. 

\section{Möbius inversion on a poset of rooted forests} \label{sec:moebius_pruning_poset}

Let $T=(\gamma,G,E)$ be a general rooted tree and note that, if not stated otherwise, all definitions and properties of this section carry over to fragmentation trees $\Ts$. For a given subset $H$ of $E$, we denote by $T-H$ the \textit{rooted forest} obtained from $T$ by deleting all edges $e\in H$; we speak of these edges as \textit{cut edges} (see Figure \ref{fig:pruned_edges}). The remaining \textit{connected components} (or \textit{components}) of $T$ are disjoint rooted trees, where the root in each component is the unique vertex that is minimal with respect to $\<$. For all $\alpha \in G$, we now denote by $\Tp{\alpha}{H}$ the subtree in $T-H$ that consists of $\alpha$ and all its descendants, see Figure~\ref{fig:subtrees}.  By slight abuse of notation, we abbreviate the corresponding vertex and edge sets by $\G{\alpha}{H}:=G(\Tp{\alpha}{H})$ and $E_\alpha(H):=E(\Tp{\alpha}{H})$, respectively; note that  $\G{\alpha}{\varnothing} = \I{\alpha} \cap G$. The rooted (fragmentation) forest $T-H$ then is the disjoint collection of all $\Tp{\alpha}{H}$ with $\alpha=\gamma$ or $\alpha$ an upper end of some $e\in H$ (c.f. Figure~\ref{fig:rooted_forest}). 

Note that, for every $H\subseteq E$ and $\alpha\in G$, the fragmentation $\Tps{\alpha}{H}$ has phantom node  $r_{\alpha}$ (where we set $r_\gamma:=r$ for consistency) and contains information about the fragments in $\S_{}^{\I{\alpha}}:=\bigcup_{R\in\Rr{\Tp{\alpha}{H}}\cup\varnothing} \L{R}{\I{\alpha}}$. 

For a given forest $T-H$, a special role is played by the subtree $\Tp{\gamma}{H}$, whose root coincides with the root of $T$. We call this tree the \textit{stump tree} of the rooted forest and say its vertex set $\G{\gamma}{H}$ is the \textit{stump set}. Obviously, the set of all \textit{stump sets} coincides with $\Rr{T}$ from \eqref{def:R(T)}. From now on, we will therefore speak of $\Rr{T}$ as the set of all stump sets.

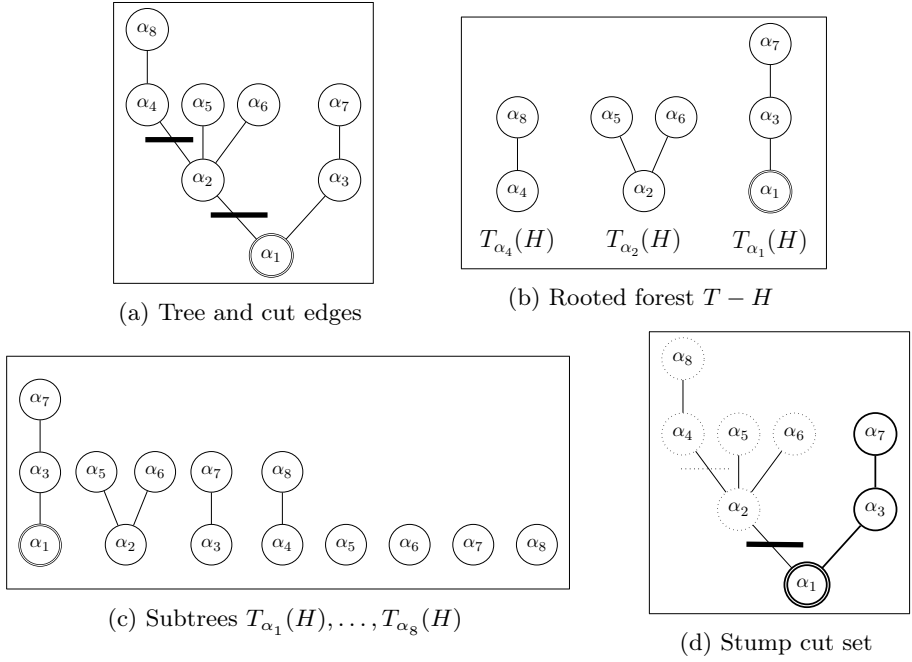
\begin{figure}[ht]
\centering
\begin{subfigure}[B]{.27\textwidth}
\centering
\resizebox{.85\linewidth}{!}{
\fbox{
\begin{tikzpicture}[grow'=up,
 state/.style={circle, draw, minimum width=.82cm, fill=white, minimum height=.7cm}
 ,level distance=1.5cm, text centered]3
\tikzstyle{leer}=[circle, minimum height=.7cm] 
\node (x) [leer] {\textcolor{white}{$\alpha_1$}}
child { [sibling distance=2.7cm]
 node (a1) [state, below=2cm, accepting] {$\alpha_1$} edge from parent[draw=none]
    child{ [sibling distance=1.1cm]
        node (a2) [state] {$\alpha_2$}
        	child{  node (a4) [state] {$\alpha_4$}
				child { node (a8) [state] {$\alpha_8$}}    		
        		}
        	child{ node (a5) [state] {$\alpha_5$}}
        	child{ node (a6) [state] {$\alpha_6$}
        	}
        	}
    child{
    	node (a3) [state] {$\alpha_3$}
    	    child{ node (a7) [state] {$\alpha_7$}}    	
        }
        }
        ;
\draw[line width=3pt] ([yshift=-.7em,xshift=-.5em]a5.south) -- ([yshift=-.7em,xshift=-.1em]a4.south);
\draw[line width=3pt] ([yshift=-.7em,xshift=.4em]a2.south) -- ([yshift=-.7em,xshift=-3.7em]a3.south);
\end{tikzpicture}
}}
\caption{\label{fig:pruned_edges}Tree and cut edges}
\end{subfigure}
\hspace{.5cm}
%
%
\begin{subfigure}{.32\textwidth}
\resizebox{\linewidth}{!}{
\fbox{
\begin{tikzpicture}[grow'=up,
 state/.style={circle, draw, minimum width=.82cm, fill=white, minimum height=.7cm}
 ,level distance=1.5cm, text centered]
\tikzstyle{leer}=[circle, minimum height=.7cm] 
\tikzstyle{l}=[rectangle, minimum height=.1cm] 
\node (z2)[leer] [] {\textcolor{white}{$\alpha_1$}}
child{ [sibling distance=1.3cm] node (c4) [state, below=3cm] {$\alpha_4$} edge from parent[draw=none]
				child { node (c8) [state] {$\alpha_8$}}    		
        		};
\node (T4) [l] [below=.2cm of c4] {\Large $\Tp{\alpha_4}{H}$};

 \node (z1)[leer] [right=1.7cm of z2] {\textcolor{white}{$\alpha_1$}}
child{ [sibling distance=1.3cm] node (c2) [state, below=3cm] {$\alpha_2$} edge from parent[draw=none]
				child { node (c5) [state] {$\alpha_5$}}
				child { node (c6) [state] {$\alpha_6$}	
				}     		
        		};
\node (T2) [l] [below=.2cm of c2] {\Large $\Tp{\alpha_2}{H}$};        		   		
\node (z) [leer] [right=1.7cm of z1]{\textcolor{white}{$\alpha_1$}}
child {
 node (c1) [state, below=3cm, accepting] {$\alpha_1$} edge from parent[draw=none]
    child{
    	node (c3) [state] {$\alpha_3$}
    	    child{ node (c7) [state] {$\alpha_7$}}    	
        }
        };
\node (T1) [l] [below=.2cm of c1] {\Large $\Tp{\alpha_1}{H}$}; 
\end{tikzpicture}
}}
\caption{\label{fig:rooted_forest}Rooted forest $T-H$}
\end{subfigure}
\\
\begin{subfigure}[B]{.49\textwidth}
\resizebox{\linewidth}{!}{
\fbox{
\begin{tikzpicture}[grow'=up,
 state/.style={circle, draw, minimum width=.82cm, fill=white, minimum height=.7cm}
 ,level distance=1.5cm, text centered]
\tikzstyle{leer}=[circle, minimum height=.7cm] 
\tikzstyle{l}=[rectangle, minimum height=.1cm] 
\node (x1) [leer] {\textcolor{white}{$\alpha_1$}}
child { [sibling distance=1cm]
 node (a1) [state, below=2cm,accepting] {$\alpha_1$}
 edge from parent[draw=none]
				child { node [state] {$\alpha_3$}
				child { node (t) [state] {$\alpha_7$}
				}};    		
       };
\node (x2) [leer, right=.9cm of x1] {\textcolor{white}{$\alpha_1$}}
child { [sibling distance=1.2cm]
 node  (a2) [state, below=2cm] {$\alpha_2$}
 edge from parent[draw=none]
				child { node  [state] {$\alpha_5$}} 
				child { node  [state] {$\alpha_6$}}     		
       };
\node (x3) [leer, right=.9cm of x2] {\textcolor{white}{$\alpha_1$}}
child { [sibling distance=1cm]
 node (a3) [state, below=2cm] {$\alpha_3$}
 edge from parent[draw=none]
				child { node [state] {$\alpha_7$}}    		
       };
\node (x4) [leer, right=.6cm of x3] {\textcolor{white}{$\alpha_1$}}
child { [sibling distance=1cm]
 node (a4) [state, below=2cm] {$\alpha_4$}
 edge from parent[draw=none]
				child { node [state] {$\alpha_8$}}    		
       };
\node (x5) [leer, right=.45 of x4] {\textcolor{white}{$\alpha_1$}}
	child { [sibling distance=1cm]
	 node  (a5) [state, below=2cm] {$\alpha_5$}
 	 edge from parent[draw=none]
 	 };
\node (x6) [leer, right=.45 of x5] {\textcolor{white}{$\alpha_1$}}
	child { [sibling distance=1cm]
	 node  (a6) [state, below=2cm] {$\alpha_6$}
 	 edge from parent[draw=none]
 	 };
\node (x7) [leer, right=.45cm of x6] {\textcolor{white}{$\alpha_1$}}
	child { [sibling distance=1cm]
	 node (a7) [state, below=2cm] {$\alpha_7$}
 	 edge from parent[draw=none]
 	 };
\node (x8) [leer, right=.45cm of x7] {\textcolor{white}{$\alpha_1$}}
	child { [sibling distance=1cm]
	 node (a8) [state, below=2cm] {$\alpha_8$}
 	 edge from parent[draw=none]
 	 };
\node [above=.09cm of t] {}
;
\node [below=.09cm of a1] {}
;
\end{tikzpicture}
}}
\caption{\label{fig:subtrees}Subtrees $\Tp{\alpha_1}{H},\ldots ,\Tp{\alpha_8}{H}$}
\end{subfigure}
\hspace{.7cm}
\begin{subfigure}{.23\textwidth}
\resizebox{\linewidth}{!}{
\fbox{
\begin{tikzpicture}[grow'=up,
 state/.style={circle, draw, minimum width=.82cm, fill=white, minimum height=.7cm}
 ,level distance=1.5cm, text centered, 
 emph/.style={edge from parent/.style={dotted,draw}},
 fat/.style={edge from parent/.style={draw, thick}}]
\tikzstyle{leer}=[circle, minimum height=.7cm] 
\node (x) [leer] {\textcolor{white}{$\alpha_1$}}
child [line width=.3pt] { [sibling distance=2.7cm]
 node (a1) [state, below=2cm, accepting, line width=1pt] {$\alpha_1$} edge from parent[draw=none]
    child [line width=.3pt]{ [sibling distance=1.1cm]
        node (a2) [state, dotted] {$\alpha_2$}
        	child{  node (a4) [state, dotted] {$\alpha_4$}
				child { node (a8) [state, dotted] {$\alpha_8$}}   		
        		}
        	child{ node (a5) [state, dotted] {$\alpha_5$}}
        	child{ node (a6) [state, dotted] {$\alpha_6$}
        	}
        	}
    child [line width=1pt]{
    	node (a3) [state] {$\alpha_3$}
    	    child [line width=1pt] { node (a7) [state] {$\alpha_7$}}    	
        }
        }
        ;
\draw[dotted] ([yshift=-.7em,xshift=-.5em]a5.south) -- ([yshift=-.7em,xshift=-.1em]a4.south);
\draw[line width=3pt] ([yshift=-.7em,xshift=.4em]a2.south) -- ([yshift=-.7em,xshift=-3.7em]a3.south);
\end{tikzpicture}
}}
\caption{Stump cut set\label{fig:stump_set}}
\end{subfigure}
\caption{\label{fig:tree_subtrees}(a) Tree $T=(\gamma,G,E)$ with root $\gamma=\alpha_1$, vertex set $G=\{\alpha_1,\ldots,\alpha_{8}\}$, and set of pruning edges $H=\{(\alpha_1,\alpha_2),(\alpha_2,\alpha_4)\}$. (b) The forest $T-H$ obtained from the tree in (a). The stump tree is $\Tp{\alpha_1}{H}$ with $R=\G{\gamma}{H}=\{\alpha_1,\alpha_3,\alpha_7\}$. The root of the stump tree is indicated by a double circle since it coincides with the root of $T$. (c) Collection of all subtrees $\Tp{\alpha}{H}$, $\alpha\in G$, in the forest $T-H$; $T$ and $H$ from (a). (d) The stump cut set for the stump set in (b) is $\partial(R)=\{(\alpha_1,\alpha_2)\}$. The stump set and the stump cut set are indicated in bold.
}
\end{figure}

Any stump set may be defined via a special set of cut edges. For a given $R \in \Rr{T}$, we denote by $\partial(R)$ the set of edges that separates   $R$ from the remaining set of vertices $G\setminus R$ and call it the \textit{stump cut set} of $R$, compare Figure~\ref{fig:stump_set}. Explicitly,
\begin{equation*}
\partial(R):=\{(\alpha,\beta)\in E: \alpha\in R,\, \beta\in G\setminus R\};
\end{equation*}
in particular, $\partial(G) = \varnothing$. The set of all stump cut sets is
\begin{equation}\label{def:stump_cut_set}
\Cc{T} := \{ \partial(R): R \in \Rr{T}\}.
\end{equation}
Obviously, $\Cc{T}$ is the set of all antichains of the poset $(E,\<)$, that is, 
\begin{equation}\label{antichain}
\Cc{T}=\{H\subseteq E \mid\text{ for all } e_1,e_2\in H,  e_1 \neq e_2, \text{ one has } e_1\not\preccurlyeq e_2 \text{ and } e_2  \not\preccurlyeq e_1\}.
\end{equation}
We can thus rewrite \eqref{def:R(T)} as
\begin{equation}\label{eq:root_in_G_gamma}
\Rr{T}=\{\G{\gamma}{C} : C\in \Cc{T}\}.
\end{equation}

%
%

%

\begin{fact} \label{fact:subtrees}
For every $H\subseteq E$, the components of $T-H$  have the following properties: 
\begin{enumerate}[label=(\Alph*)]
\item \label{fact:subtrees_3}$\big(\Tp{\gamma}{H}\big)_{\alpha}(K)=\Tp{\alpha}{H\cup K}$, $\alpha\in \G{\gamma}{H}$, $K\subseteq E_\gamma(H)$
\item \label{fact:subtrees_2}$\Tp{\alpha}{H}=\Tp{\alpha}{H\cup C}$ for $C\in\Cc{\Tp{\gamma}{H}}$ and $\alpha\notin \G{\gamma}{H\cup C}$.
\end{enumerate}
These properties carry over to the corresponding vertex sets of the rooted trees.
\end{fact}

\begin{proof}
\ref{fact:subtrees_3} is due to a general property of graph decomposition  via recursive edge deletion: the order in which edges are deleted does not affect the final object.
So $(T-H)(K) = T-(H \cup K)$ for all $H,K \subseteq E$; in particular, the stump tree is the same in both cases. 
\ref{fact:subtrees_2} For every $C\in\Cc{\Tp{\gamma}{H}}$ and $\alpha\notin \G{\gamma}{H\cup C}$, we have $C \cap E_\alpha(\varnothing)=\varnothing$ due to the antichain property \eqref{antichain} of $\Cc{T}$.
But a subtree $\Tp{\alpha}{H}$ is not affected by  deletion of an edge $e\notin E_\alpha(\varnothing)$.
\end{proof}

\subsection{Pruning Poset}\label{sec:pruning_poset}

From now on, let $T=(\gamma,G,E)$ be fixed and let us investigate the set of all subsets of edges of $T$, denoted by $\Pot{E}$. We introduce a partial order $\pD$ on $\Pot{E}$ and say that $H\pD K$ for any two sets of cut edges $H,K\subseteq E$ when $H=K\cup A$  with $A \subseteq E_\gamma(K)$. In words, $H\pD K$ whenever the additional cuts in $H\setminus K$ occur in the stump tree of the rooted forest $T-K$. 
The set $\Pot{E}$ along with the partial order $\pD$ constitutes a poset $\PP(T):=(\Pot{E},\pD)$. Since the cut edges prune the tree (in an intuitive way of thinking), we call $\PP(T)$ the \textit{pruning poset} of $T$. A specific example with corresponding Hasse diagram is shown in Figure~\ref{fig:poset_tree_decompositions}. For every $K\subseteq E$, we clearly have the isomorphic relation
\begin{equation}\label{eq:isomorphism_D}
(\{H: H\pD K\}, \pD ) \simeq \PP(\Tp{\gamma}{K}).
\end{equation}

$\PP(T)$ has a maximal element $\varnothing$, which means that $H\pD \varnothing$ for all $H\subseteq E$; but in general no minimal element $\mathbf 0$, with $\mathbf 0\pD H$ for all $H\subseteq E$. 
As a consequence, $\PP(T)$ is, in general, not a lattice.
Nonetheless, every embedded subposet or \emph{interval}
\[
[H,K]_P:=\big(\{I\subseteq E : H\! \pD \! I\! \pD K\},\pD\! \big),\quad H\pD K
\]
of $\PP(T)$ is a lattice. We omit the subscript in what follows. Due to   \eqref{eq:isomorphism_D}, we conclude the isomorphic relation $[H,K]\simeq [H\setminus K, \varnothing]$ for any $H\pD K$. It is therefore sufficient to investigate the properties of $[H,\varnothing]$ for every $H\subseteq E$. The interval $[H,\varnothing]$ obviously has maximal element $\varnothing$ and minimal element $H$. Every path (top to bottom) in $[H,\varnothing]$ represents the possibility to add elements from $H$ in nonincreasing order with respect to $\<$.

\begin{figure}
\resizebox{\textwidth}{!}{
\begin{tikzpicture}[node distance= .5cm, bend angle=0,auto]
\tikzstyle{circ}=[draw, circle, left color=white, minimum width=.1cm,  minimum height=.1cm]
\tikzstyle{leer}=[circle, minimum height=.1cm]
\tikzstyle{rec}=[draw, rectangle, black, left color=white, minimum width=.5cm]
\tikzstyle{recl}=[draw, rectangle, black, left color=white, minimum width=.5cm]
\tikzstyle{ll}=[line width=1.5pt]
\node [circ, accepting] (g) {$\,$};
\node [leer] (y0) [above = .7cm of g] {$\,$};
\node [circ] (a4) [right =.3cm of y0] {$\,$}
edge node[right=.2cm, pos=0.6]{$e_4$} (g); 
\node [circ] (a3) [left=1.1cm of a4] {$\,$}
edge node[left=.2cm,pos=0.6]{$e_3$} (g);
\node [leer] (y1) [above = .7cm of a3] {$\,$};
\node [circ] (a2) [right =.3cm of y1] {$\,$}
edge node[right=.2cm, pos=0.6]{$e_2$} (a3); 
\node [circ] (a1) [left=1.1cm of a2] {$\,$}
edge node [left]{$e_1$} (a3);
\node [leer] (p) [above= 2.1cm of a4] {$\,$};
\node [recl] (0) [right=8cm of p] {$\varnothing$};
\node [leer] (x0) [below=.8cm of 0] {};
\node [recl] (2) [left =1cm of x0] {$e_2$}
edge [ll] node{} (0);
\node [recl] (1) [left=1cm of 2] {$e_1$}
edge [ll] node{} (0);
\node [recl] (4) [right=1cm of x0] {$e_4$}
edge [ll] node{} (0);
\node [rec] (3) [right=1cm of 4] {$e_3$}
edge node{} (0);
\node [leer] (x1) [below=1cm of x0] {};
\node [recl] (14) [left =.3cm of x1] {$e_1,e_4$}
edge [ll] node{} (1)
edge [ll] node{} (4);
\node [recl] (12) [left=1cm of 14] {$e_1,e_2$}
edge [ll] node{} (1)
edge [ll] node{} (2);
\node [rec] (13) [left=1cm of 12] {$e_1,e_3$}
edge node{} (1);
\node [rec] (23) [right=.3cm of x1] {$e_2,e_3$}
edge node{} (2);
\node [recl] (24) [right=1cm of 23] {$e_2,e_4$}
edge [ll] node{} (2)
edge [ll] node{} (4);
\node [rec] (34) [right=1cm of 24] {$e_3,e_4$}
edge node{} (3)
edge node{} (4);
\node [leer] (x2) [below=1cm of x1] {};
\node [recl] (123) [left =.7cm of x2] {$e_1,e_2,e_3$}
edge [ll] node{} (12);
\node [rec] (134) [left=1cm of 123] {$e_1,e_3,e_4$}
edge node{} (13)
edge node{} (14);
\node [recl] (124) [right=.7cm of x2] {$e_1,e_2,e_4$}
edge [ll] node{} (12)
edge [ll] node{} (24)
edge [ll] node{} (14);
\node [rec] (234) [right=1cm of 124] {$e_2,e_3,e_4$}
edge node{} (23)
edge node{} (24);
\node [recl] (1234) [below=.8cm of x2] {$e_1,e_2,e_3,e_4$}
edge [ll] node{} (124)
edge [ll] node{} (123);
\draw[thick,<->,>=stealth] ([yshift=0.5em,xshift=2em]a4.south) -- ([yshift=.9em,xshift=-3.5em]13.south); 
\end{tikzpicture}
}
\caption{\label{fig:poset_tree_decompositions}Left: Tree $T$ with edges $\{e_1,\ldots,e_4\}$. Right: Hasse diagram for the pruning poset $\PP(T)$ of $T$. The subposet $[H,\varnothing]$ is indicated in bold.}
\end{figure}
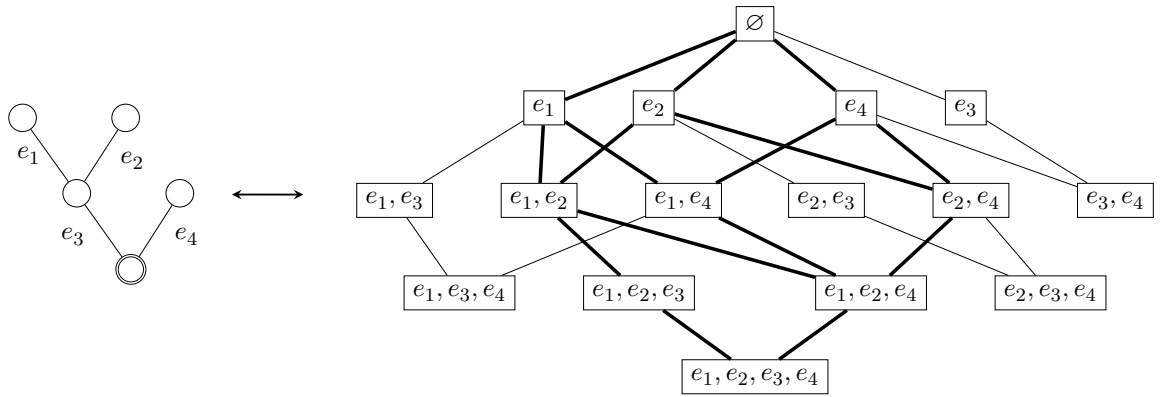

\begin{fact}\label{remark:poset_ideals}
It is clear by construction that every $I\in[H,\varnothing]$ is an ideal of $(H,\>)$, the poset $(H,\>)$ of edges with reversed partial order. Since the elements in $[H,\varnothing]$ are, in addition, ordered by reversed set inclusion, one has 
\begin{equation*}\label{eq:poset_of_ideals}
[H,\varnothing]=\big(\{I: I \text{ ideal of }(H,\>)\},\supseteq\big),
\end{equation*}
where $\supseteq$ denotes reversed set inclusion.
\end{fact}

We can now state the Möbius function for the pruning poset, see \cite[Chap.~4]{Aigner} or \cite[Chap.~3]{Stanley1} for an introduction into the topic.

\begin{proposition}\label{thm:Moebius_function_D}
For a given tree $T=(\gamma,G,E)$, the M\"obius function  for the pruning poset $\PP(T)$ is, for every $H,K\subseteq E$ with $H\pD K$, given by
\begin{equation}\label{eq:Moebius_function_D}
\mu(H,K)=
\begin{cases}
(-1)_{}^{|H|-|K|}, &\text{ if } H\setminus K\in \Cc{\Tp{\gamma}{K}},\\
0, &\text{otherwise}.
\end{cases}
\end{equation}
\end{proposition}

\begin{proof}
Due to the  representation of $[H,\varnothing]$ as an ideal in Fact~\ref{remark:poset_ideals} and the antichain property \eqref{antichain} of stump cut sets, the claim follows from a very general result for M\"obius functions of lattices of the form $(\{I: I \text{ ideal of } P\},\subseteq)$, $P$ any poset; see Example~3.9.6 in \cite{Stanley1}. 
\end{proof}

Now that we have an explicit expression for the M\"obius function, we can use M\"obius inversion (see \cite{Rota} or \cite[Prop.~4.18]{Aigner}) on $\PP(T)$, which for any two functions $f,g\colon \Pot{E}\to \mathbb R$ and any subset $K\subseteq E$ reads
\begin{equation}\label{eq:Moebius_inversion_D(T)}
g(K)=\sum\limits_{H\pD K} f(H) \ \Leftrightarrow \ f(K)= \sum\limits_{H\pD K}  \mu(H,K)\ g(H). 
\end{equation}

So far, we focussed on the set of all possible subsets of edges of a given tree $T=(\gamma,G,E)$. Let us now shift the perspective to the set of all rooted forests that can be obtained from $T$ by edge deletion. Obviously, there is a one-to-one correspondence between the elements of $\{T-H:H\subseteq E\}$ and those of $\Pot{E}$. We may thus define the poset $\mathbb F(T):=(\{T-H:H\subseteq E\},\pF)$ by specifying that $T-H\pF T-K$ precisely if $H\pD K$ for $H,K\subseteq E$. It is then clear that $\mathbb F(T)$ is isomorphic to $\PP(T)$ by construction. All properties of $\PP(T)$, such as isomorphism, the Möbius function in \eqref{eq:Moebius_function_D}, as well as the Möbius inversion formula in \eqref{eq:Moebius_inversion_D(T)} therefore carry over to $\mathbb F(T)$. 

As mentioned in the Introduction, the poset $\mathbb F(T)$ is a special case of the \textit{poset of planar forests} introduced by Foissy \cite{Foissy}, restricted to what he calls transformations of the second kind and applied to the stump tree only. Foissy also uses M\"obius inversion on his more general poset of planar forests. He calculates the M\"obius function for small examples, but does not give a general formula. Fortunately, our special case has enough structure to allow for a simple, general and explicit result. This will be the key to  an explicit expression for the tree probabilities in the context of the fragmentation process, which we consider next.

\section{Tree probabilities via Möbius inversion} \label{sec:moebius_fragmentation_trees}

Consider a fragmentation tree $\Ts=(\gamma,G, E,L)$. Let $\Gamma:= G\cup \S$  and assign to every element $s \in \Gamma$  some event (in the sense of a finite  set) $\Ev(s)$. We will throughout abbreviate $\Ev(\{\alpha\})=:\Ev(\alpha)$ and $\textstyle \bigcup_{J\in \L{\G{\alpha}{H}}{\I{\alpha}}} \Ev(J)=:\Ev\big(\L{\G{\alpha}{H}}{\I{\alpha}}\big)$ for  $\alpha\in G$, $H \subseteq E$. At this point, we neither give a meaning nor a law to the events, but will assume that the events are \textit{nested} according to the set structure, i.e., that 
\begin{equation} \label{cond:nested_events}
\Ev(s_1) \subseteq \Ev(s_2) \text{ if and only if } s_1 \subseteq s_2 \subseteq \Gamma.
\end{equation}
Note that in general $\Ev(s_1)\cup \Ev(s_2)\neq \Ev(s_1 \cup s_2)$, in particular $\Ev(I'_\alpha) \cup \Ev(\alpha) \cup \Ev(I''_\alpha)  \subseteq \Ev(I^{}_\alpha)$, but equality need not hold.
Let  $\Xi$ be the set generated from $\{\Ev(s): s \in \Gamma\}$ by arbitrary unions and set exclusions. The event $\Ev(\I{\alpha})\setminus \Ev\big(\L{\G{\alpha}{H}}{\I{\alpha}}\big)$ will often be required. Let us state the following fact.

\begin{fact}\label{fact:nested_events}
For events nested according to \eqref{cond:nested_events} we have $\Ev(\L{A}{})\subseteq \Ev(\L{G}{})$ for $G\subseteq A\subseteq L$. Moreover, for every $\alpha\in G$, $H\subseteq E$, the following properties hold:
\begin{enumerate}[label=(\Alph*)]
\item \label{fact:nested_events_1} $ \Ev\big(\L{\G{\alpha}{H}}{\I{\alpha}}\big)\subseteq \Ev\big(\L{\G{\alpha}{H\cup K}}{\I{\alpha}}\big)\subseteq \Ev\big(\L{\varnothing}{\I{\alpha}}\big)=\Ev(\I{\alpha})$ for $K\subseteq E$.
\item \label{fact:nested_events_2} $\Ev(\beta) \subseteq \Ev(\I{\alpha})\setminus \Ev\big(\L{\G{\alpha}{H}}{\I{\alpha}}\big)$ for all $\beta\in \G{\alpha}{H}$.
\item \label{fact:nested_events_3} $\Ev\big(\L{\G{\beta}{H}}{\I{\beta}}\big)\subseteq \Ev\big(\L{\G{\alpha}{H}}{\I{\alpha}}\big)$ for all $\beta\in \G{\alpha}{H}$.
\end{enumerate}
\end{fact}

\begin{proof}
Let  $G\subseteq A\subseteq L$. By definition of  $\L{A}{}$ and  $\L{G}{}$,  for any $J\in \L{A}{}$ there is an $I\in\L{G}{}$ such that $J\subseteq I$ and thus $\Ev(J)\subseteq \Ev(I)$ by \eqref{cond:nested_events}. \ref{fact:nested_events_1} follows from the latter statement since $\varnothing\subseteq \G{\alpha}{H\cup K}\subseteq \G{\alpha}{H}$ for any $\alpha\in G$, $H,K\subseteq E$, and because $\L{\varnothing}{\I{\alpha}}=\{\I{\alpha}\}$. \ref{fact:nested_events_2}: Let $\beta\in\G{\alpha}{H}$ for some $\alpha \in G$. Since $\beta\in \I{\beta}\subseteq \I{\alpha}$, we know $\Ev(\beta)\subseteq \Ev(\I{\beta})\subseteq \Ev(\I{\alpha})$ by \eqref{cond:nested_events}. On the other hand, $\L{\G{\alpha}{H}}{\I{\alpha}} \setminus \varnothing$ is a partition of  $\I{\alpha}\setminus \G{\alpha}{H}$, so  $\beta\notin J$ for any $J\in\L{\G{\alpha}{H}}{\I{\alpha}}$ and thus $\Ev(\beta)\nsubseteq \Ev\big(\L{\G{\alpha}{H}}{\I{\alpha}}\big)$ by \eqref{cond:nested_events}. \ref{fact:nested_events_3} follows from \eqref{cond:nested_events} and the fact that $\L{\G{\beta}{H}}{\I{\beta}}\subseteq \L{\G{\alpha}{H}}{\I{\alpha}}$ for all $\beta\in \G{\alpha}{H}$.
\end{proof}

Let now $\t{}:\Xi\to\mathbb R_{\geqslant 0}$ be a function that assigns a scalar to each event in $\Xi$. Later, $\t{}$  will turn into the waiting time for the  event and thus generalise the previous meaning of $\t{}$, but here we are not tied to an underlying process. Let us write $\t{\mathcal G}:=\t{}(\mathcal G)$ and assume that 
\begin{equation}\label{cond:sigma}
\t{\mathcal G}\leqslant \t{\mathcal H} \text{ if and only if } \mathcal G\supseteq \mathcal H, \quad \mathcal H, \mathcal G \in \Xi.
\end{equation}

Our object of interest in this section is the event $\Max{\Ev}{H}\cap \m{\Ev}{H}$, where 
\begin{equation}\label{def:Max_G}
\Max{\Ev}{H}:=\big\{\max\{\t{\Ev(\alpha)}:\alpha\in \G{\gamma}{H}\}\leqslant t,\, t < \t{\Ev(\L{\G{\gamma}{H}}{})}\big\}, \quad G \subseteq L, \quad t\in \NN,
\end{equation}
and
\begin{equation}\label{def:m_G}
\m{\Ev}{H}:=\bigcap_{\alpha\in G} \big \{\t{\Ev(\alpha)}=\t{\Ev(\I{\alpha})\setminus \Ev(\L{\G{\alpha}{H}}{\I{\alpha}})}\big\}, \quad H \subseteq E.
\end{equation}
We will see later that $\Max{\Ev}{H}\cap \m{\Ev}{H}$ generalises the tree event in \eqref{def:tree_event_tau}. Let us only mention here that \eqref{def:m_G} may be understood as an order relation within each of the connected components of $\Ts-H$. Our aim is to express $\Max{\Ev}{H}\cap \m{\Ev}{H}$ in terms of a collection of certain minima combined with order relations, via an inclusion-exclusion principle. The order relations are those just defined, and the minima are analogous to the maxima, namely
\begin{equation}\label{def:Min_G(t)}
\Min{\Ev}{H}:=\big\{\min\{\t{\Ev(\alpha)}:\alpha\in \G{\gamma}{H}\}\leqslant t,\, t<\t{\Ev(\L{\G{\gamma}{H}}{})}\big\}, \quad H \subseteq L, \quad t\in \NN.
\end{equation}

We will proceed in the opposite direction and start with a decomposition of the joint event of the form $\Min{\Ev}{H}\cap \m{\Ev}{H}$ into a collection of maxima and then apply Möbius inversion on $\PP(T)$ from \eqref{eq:Moebius_inversion_D(T)}. Anticipating that the stump set will play a special role in our final tree probabilities, we formulate the following lemma.

\begin{lemma}\label{lemma:Min_Max}
Let $\Ts=(\gamma,G,E,L)$ be a fragmentation tree and $K\subseteq E$. If \eqref{cond:nested_events} and \eqref{cond:sigma} are satisfied, then
\begin{equation}\label{eq:Min_in_Max}
\P\big(\Min{\Ev}{K}\cap \m{\Ev}{K}\big) = \sum\limits_{C\in \Cc{\Tp{\gamma}{K}}} \P\big(\Max{\Ev}{K\cup C}\cap \m{\Ev}{K\cup C}\big),
\end{equation}
where $\P$ denotes a probability measure on $\Xi$ and $\Cc{\Tp{\gamma}{K}}$ is from \eqref{def:stump_cut_set}.
\end{lemma}

\begin{proof}
We will decompose the probability for the joint event $\Min{\Ev}{K}\cap\m{\Ev}{K}$ part by part. We first express the minimum in $\Min{\Ev}{K}$ in terms maxima using the well-known disjoint decomposition, which here reads
\begin{align}\label{eq:Min=Summe_Max}
\begin{split}
\{\min\{\t{\Ev(\alpha)}:\alpha\in\G{\gamma}{K}\}\leqslant t\}&\\
&\hspace{-5cm} = \bigcup_{\varnothing \neq A \subseteq \G{\gamma}{K}} \{\max\{\t{\Ev(\alpha)}:\alpha\in A\}\leqslant t,\, t<\min\{\t{\Ev(\beta)} :\beta\in \G{\gamma}{K}\setminus A\}\}.
\end{split}
\end{align}
We now intersect both sides of \eqref{eq:Min=Summe_Max} with $\m{\Ev}{K}$ and then evaluate the probability. Since 
\begin{equation}\label{eq:mK_implies}
\m{\Ev}{K} \subseteq \{\t{\Ev(\alpha)}\leqslant \t{\Ev(\beta)} \text{ for all } \alpha\in G \text{ and } \beta\in \G{\alpha}{K}\}
\end{equation}
by Fact~\ref{fact:nested_events}~\ref{fact:nested_events_2}, we have
\[
\P\big(\big\{\max\{\t{\Ev(\alpha)}:\alpha\in A\}\leqslant t,\, t<\min\{\t{\Ev(\beta)}:\beta\in \G{\gamma}{K}\setminus A\}\big\}\cap\m{\Ev}{K}\big)=0
\]
for every subset $A\subseteq G$ that does not contain the root, or is not contiguous with respect to the partial order on $\Tps{\gamma}{K}$, that is, if $A$ is not a stump set of $\Tps{\gamma}{K}$. Using \eqref{eq:mK_implies} once more, we conclude that
\[
\big\{\min\big\{\t{\Ev(\beta)}:\beta\in \G{\gamma}{K}\setminus R\big\}\big\} \cap \m{\Ev}{K}=\big\{\min\big\{\t{\Ev(\beta)}:\beta\in M(\G{\gamma}{K}\setminus R)\big \}\big\} \cap \m{\Ev}{K};
\]
recall that $M(\G{\gamma}{K}\setminus R)$ is the set of vertices in $\G{\gamma}{K}\setminus R$ that are minimal with respect to $\<$. We may thus write
\begin{align*}\label{eq:Min=Summe_Max2}
\begin{split}
\P\big(\Min{\Ev}{K}& \cap \m{\Ev}{K} \big)  =\! \!\sum\limits_{R\in \Rr{\Tps{\gamma}{K}}}\!\!\!\P\big(\big\{\max\{\t{\Ev(\alpha)}:\alpha\in R\}\leqslant t,\, t<\t{\Ev(\L{\G{\gamma}{K}}{})},\\[.7em]
& \qquad\quad\qquad\qquad t<\min\{\t{\Ev(\beta)} :\beta\in M(\G{\gamma}{K}\!\setminus\! R)\}\}\cap \m{\Ev}{K} \big)\\[.7em]
& = \!\!\sum\limits_{R\in \Rr{\Tps{\gamma}{K}}} \!\! \!\P\big(\big\{\max\{\t{\Ev(\alpha)}:\alpha\in R\}\leqslant t<\t{\Ev(\L{\G{\gamma}{K}}{})},\\[-1.0em]
& \qquad\quad \qquad\qquad t<\min\big\{\t{\!\Ev(\I{\beta})\setminus \Ev(\L{\G{\beta}{K}}{\I{\beta}})}\!\!:\beta\in M(\G{\gamma}{K}\!\setminus\! R)\big\}\big\}\cap \m{\Ev}{K}\big)\\[.7em]
&  =\!\! \sum\limits_{R\in \Rr{\Tps{\gamma}{K}}}\!\!\! \P\big(\Max{\Ev}{K}\cap \m{\Ev}{K}\big).
\end{split}
\end{align*}
The second equality is due to the intersection with $\m{\Ev}{K}$, see \eqref{def:m_G}. In the third step, we used that 
\begin{equation*}\label{eq:intersection:events}
\Ev(\L{\G{\gamma}{K}}{})\bigcup_{\beta\in M(\G{\gamma}{K}\setminus R)} \Ev(\I{\beta})\setminus \Ev(\L{\G{\beta}{K}}{\I{\beta}})=\Ev(\L{R}{})\,,
\end{equation*}
which follows by \eqref{eq:props_I_alpha} applied to the stump tree $\Tps{\gamma}{K}$ with the help of Fact~\ref{fact:subtrees}~\ref{fact:subtrees_3}. Altogether this gives 
\begin{align}\label{eq:Min=Summe_Max3}
\begin{split}
\P\big(\Min{\Ev}{K}\cap \m{\Ev}{K} \big)
& = \sum\limits_{C\in \Cc{\Tp{\gamma}{K}}} \P\big(\Max{\Ev}{K\cup C}\cap \m{\Ev}{K}\big)
\end{split}
\end{align}
due to \eqref{eq:root_in_G_gamma} and Fact~\ref{fact:subtrees}~\ref{fact:subtrees_3}.
Let us finally consider the ordering relation $\m{\Ev}{K}$ in the joint event on the right-hand side of \eqref{eq:Min=Summe_Max3}. Consider first an $\alpha\notin \G{\gamma}{K\cup C}$, in which case we obtain $\Ev(\I{\alpha})\setminus\Ev\big(\L{\G{\alpha}{K}}{\I{\alpha}}\big)=\Ev(\I{\alpha})\setminus\Ev\big(\L{\G{\alpha}{K\cup C}}{\I{\alpha}}\big)$ by Fact~\ref{fact:subtrees}~\ref{fact:subtrees_2}. For an $\alpha\in \G{\gamma}{K\cup C}$, we have $\Max{\Ev}{K\cup C}\subseteq\{\t{\Ev(\alpha)}<\t{\Ev(\L{\G{\gamma}{K\cup C}}{})}\}$. Since furthermore $\t{\Ev(\L{\G{\gamma}{K\cup C}}{})}\leqslant \t{\Ev(\L{\G{\alpha}{K\cup C}}{\I{\alpha}})}$ by Fact~\ref{fact:nested_events}~\ref{fact:nested_events_3}, we can conclude 
\begin{align*}
\Max{\Ev}{K\cup C}& \cap\big \{\t{\Ev(\alpha)}=\t{\Ev(\I{\alpha})\setminus \Ev(\L{\G{\alpha}{K}}{\I{\alpha}})}\big\}\\
& =\Max{\Ev}{K\cup C}\cap \big\{\t{\Ev(\alpha)}=\t{\big(\Ev(\I{\alpha})\setminus \Ev(\L{\G{\alpha}{K}}{\I{\alpha}})\big)\setminus \Ev(\L{\G{\alpha}{K\cup C}}{\I{\alpha}})}\big\}.
\end{align*}
Since furthermore 
\[
\big(\Ev(\I{\alpha})\setminus \Ev\big(\L{\G{\alpha}{K}}{\I{\alpha}}\big)\big)\setminus \Ev\big(\L{\G{\alpha}{K\cup C}}{\I{\alpha}}\big)=\Ev(\I{\alpha})\setminus \Ev\big(\L{\G{\alpha}{K\cup C}}{\I{\alpha}}\big)
\]
by Fact~\ref{fact:nested_events}~\ref{fact:nested_events_1}, we can rewrite the joint event as
\[
\Max{\Ev}{K\cup C}\cap \m{\Ev}{K} =  \Max{\Ev}{K\cup C}\cap \m{\Ev}{K \cup C}.
\]
Together with \eqref{eq:Min=Summe_Max3} this completes the proof.
\end{proof}

\begin{proposition}\label{thm:Max_in_Min}
Under the conditions of Lemma~\ref{lemma:Min_Max}, the following holds for every $K\subseteq E$:
\begin{equation*} \label{eq:Max_in_Min}
\P\big(\Max{\Ev}{K}\cap \m{\Ev}{K}\big)=\sum\limits_{H\subseteq E_\gamma(K)} (-1)_{}^{|H|}\  \P\big(\Min{\Ev}{H\cup K}\cap \m{\Ev}{H\cup K}\big).
\end{equation*}
\end{proposition}

\begin{proof}
Recall the Möbius function $\mu$ for the pruning poset $\PP(T)$ in \eqref{eq:Moebius_function_D} and rewrite it as $\mu(H,K)\, (-1)_{}^{|H|-|K|}=\mathds 1^{}_{\{H\setminus K\, \in\, \Cc{\Tp{\gamma}{K}}\}}$ for $H,K\subseteq E, H\pD K$. This allows to reformulate \eqref{eq:Min_in_Max} from Lemma~\ref{lemma:Min_Max}  as
\begin{align}\label{eq:Moebius_trees}
\begin{split}
(-1)_{}^{|K|}\ & \P\big(\Min{\Ev}{K}\cap \m{\Ev}{K}\big)\\
& \qquad\qquad  = (-1)_{}^{|K|}\, \sum\limits_{H\subseteq E} \mathds 1^{}_{\{H\setminus K\,\in\, \Cc{\Tp{\gamma}{K}}\}}\  \P\big(\Max{\Ev}{H}\cap \m{\Ev}{H}\big)\\
& \qquad\qquad = \sum\limits_{H\pD K } \mu(H,K)\ (-1)_{}^{|H|}\ \P\big(\Max{\Ev}{H}\cap \m{\Ev}{H}\big), 
\end{split}
\end{align}
where the last equality is due to isomorphism on $\PP(\Tp{\gamma}{H})$ in \eqref{eq:isomorphism_D}. 
Möbius inversion on $\PP(T)$ (cf. \eqref{eq:Moebius_inversion_D(T)}) then yields the inverse  of \eqref{eq:Moebius_trees}:
\begin{align*}
(-1)_{}^{|K|}\,\P\big(\Max{\Ev}{K}\cap \m{\Ev}{K}\big)  &= \sum\limits_{H\pD K }\, (-1)_{}^{|H|}\ \P\big(\Min{\Ev}{H}\cap \m{\Ev}{H}\big) \\
& \hspace{-1.5cm} = \sum\limits_{H\subseteq E_\gamma(K)} (-1)_{}^{|H\cup K|}\ \P\big(\Min{\Ev}{H\cup K}\cap \m{\Ev}{H\cup K}\big),
\end{align*}
where the last equality is once more isomorphism on $\PP(\Tp{\gamma}{H})$. 
\end{proof}

We can now use Proposition~\ref{thm:Max_in_Min} to evaluate the tree probabilities in \eqref{def:tree_event_tau}. To this end, we define events for the fragmentation process as $\Ef(s):=\{s\}$ for all $s\in \Gamma$, so that $\t{\Ef(s)}=\t{s}$ is the waiting time at which the first link in $s$ is removed under $(\F{t}{})_{t\in \NN}^{}$. Since $L\setminus G = \cup_{J \in \L{G}{}} J$, we then have
\begin{equation}\label{simpl_1}
\t{L\setminus G}=\min\big \{\t{J}:J\in \L{G}{}\big \}= \min\big \{\t{\Ef(J)}:J\in \L{G}{}\big \}=\t{\Ef(\L{G}{})}. 
\end{equation}
Likewise, since 
\begin{equation}\label{eq:funny_event}
\G{\alpha}{H}= \I{\alpha}\setminus\big \{ J: J \in \L{\G{\alpha}{H}}{\I{\alpha}}\big \}=\Ef(\I{\alpha})\setminus \Ef(\L{\G{\alpha}{H}}{\I{\alpha}}),
\end{equation}
one has
\begin{equation}\label{simpl_2} 
\t{\G{\alpha}{H}}=\t{\Ef(\I{\alpha})\setminus \Ef(\L{\G{\alpha}{H}}{\I{\alpha}})}. 
\end{equation}
These seemingly more complicated expressions allow us to rewrite $\{\Ftree\}$ from \eqref{def:tree_event_tau} as the generalised tree event
\begin{equation}\label{tree_event}
\{\Ftree\}=\Max{\Ef}{\varnothing}\cap \m{\Ef}{\varnothing}
\end{equation}
with $\Max{\Ef}$ and $\m{\Ef}$ as defined in \eqref{def:Max_G} and \eqref{def:m_G}, and $\Ev$ replaced by $\Ef$.

\begin{corol}\label{corollary:prob_tree}
Let $\Ts=(\gamma,G,E,L)$ and $t\in \NN$ be given. The probability that  $(\F{t}{})_{0 \leqslant t' \leqslant t}^{}$ matches $\Ts$ is then given by
\begin{align}\label{eq:prob_tree}
\P\big(\Ftree\big)&= \sum\limits_{H \subseteq E} (-1)_{}^{|H|}~ \P \big(\Min{\Ef}{H}\cap  \m{\Ef}{H}\big)\\
& = \sum\limits_{H \subseteq E} (-1)_{}^{|H|}~ \P\big(\t{\G{\gamma}{H}}\leqslant t,\, t<\t{L\setminus \G{\gamma}{H}},\ \t{\alpha}=\t{\G{\alpha}{H}}\ \forall \ \alpha\in G \big).\notag
\end{align}
\end{corol}

The probability of a fragmentation tree $\Ts$ can thus be expressed as an alternating sum over all probabilities corresponding to fragmentation forests (i.e. the disjoint collection of fragmentation trees $\Tps{\alpha}{H}$, where either $\alpha=\gamma$, or $\alpha$ is an upper end of an edge in $H$). that can be obtained from $\Ts$ by edge deletion. For every given fragmentation forest $\Ts-H$, the ordering relation may be rewritten as
\begin{equation*}
\m{\Ef}{H} = \bigcap_{\Tps{\alpha}{H}\,\in\, \Ts-H}\quad \bigcap_{\nu \in \G{\alpha}{H}} \big\{\t{\nu}=\t{\G{\nu}{H}}\big\},
\end{equation*}
which shows that the ordering is now prescribed within each component of $\Ts-H$, in contrast to \eqref{def:tree_event_tau}, which prescribes the ordering within the entire tree.
The joint event $\Min{\Ef}{H}\cap \m{\Ef}{H}$ thus means that at least one link in the stump tree has been removed until time $t$, all the links in $L \setminus \G{\gamma}{H}$ are still intact, and the events corresponding to the vertices in $G$ happen in the prescribed order within each component. 

\begin{proof}
Choosing $\Ef(s)=\{s\}$ for all $s\in \Gamma$ clearly satisfies the nesting condition \eqref{cond:nested_events}. Furthermore, choosing $\t{}$ as the waiting time for the events $\Ef$ guarantees \eqref{cond:sigma}. We may thus use Proposition~\ref{thm:Max_in_Min} and apply it to \eqref{tree_event}, that is, for $K=\varnothing$. This yields
\[
\P\big(\Ftree\big)=\P\big(\Max{\Ef}{\varnothing}\cap  \m{\Ef}{\varnothing}\big)=\sum_{H\subseteq E} (-1)^{|H|}_{}\ \P\big(\Min{\Ef}{H}\cap  \m{\Ef}{H}\big)
\]
with $\Min{\Ef}$ from \eqref{def:Min_G(t)}.
Employing \eqref{simpl_1} and \eqref{simpl_2} once more, this time in the reverse direction, completes the proof.
\end{proof}

Although Corollary~\ref{coro:sum_trees_discrete} yields a nice decomposition of the matching probability, it cannot be evaluated in an elementary manner since the law of links to be added to $\F{t}{}$ changes over time. To see this, note first that the probability that nothing happens in a given time step is
\begin{equation}\label{def:lambda}
\P\big(\F{t+1}{}=G \mid \F{t}{}=G\big)= \prod_{J\in \L{G}{}} \big(1-\rh{J}{}\big )=:\lam{G}{},\quad \rh{J}:=\sum_{\alpha\in J}\rh{\alpha}.
\end{equation}
Secondly, suppose that, in some time step, link $\gamma\notin\{1,n\}$ is removed. Then $L$ splits into the two nonempty fragments $\Ileq{\gamma}=\{\beta\in L:\beta<\gamma\}$ and $\Igeq{\gamma}=\{\beta\in L:\beta>\gamma\}$. After removal of $\gamma$, the probability that a link in $\Ileq{\gamma}$ or $\Igeq{\gamma}$ is removed is $1-\lam{\gamma}{L}=\rh{\Ileq{\gamma}}+\rh{\Igeq{\gamma}}-\rh{\Ileq{\gamma}}\cdot \rh{\Igeq{\gamma}}$, whereas before removal of $\gamma$ it is $\rh{\Ileq{\gamma}}+\rh{\Igeq{\gamma}}=1-\lam{\gamma}{L}+\rh{\Ileq{\gamma}}\cdot \rh{\Igeq{\gamma}}$. We may thus think of $1-\lam{\gamma}{L}$ as the probability for a removal in $L\setminus\{\gamma\}$ when $\Ileq{\gamma}$ and $\Igeq{\gamma}$ are independent and of $\rh{\Ileq{\gamma}}\cdot \rh{\Igeq{\gamma}}$ as the additional probability for the case that the fragments are still dependent. We generalise the idea of a decomposition into dependent and independent parts in the next section.

\subsection{The auxiliary process}\label{sec:auxiliary_process}
We now construct an auxiliary process, which is state independent, and which jointly represents all transitions of interest for the fragmentation process and a given fragmentation tree.  We then use the auxiliary process to construct realisations of $(\F{t}{})_{t \in\NN}^{}$ that are compatible with a given fragmentation tree up to time $t$ and express matching events of the fragmentation process in terms of events of the auxiliary process.

\subsubsection{Construction of the auxiliary process}
Fix a fragmentation tree $\Ts=(\gamma,G,E,L)$. We aim at a construction of a sequence of i.i.d. random variables $(\X{t}{})^{}_{t\in \NN}$ where, for all $t\in \NN$, $\X{t}{}$ will be a family $\X{t}{}=(\X{t}{J})^{}_{J\in\S{}}$, and the $\X{t}{J}$'s will have a specific dependence for the $J$'s.
We construct this collection for every $t\in \NN$ inductively, starting with the (full or empty) external fragments of the tree and proceeding in a top-down manner.

For the start, let $t> 0$ be fixed and define $\X{t}{J}$ for each external fragments $J\in\L{G}{}$ independently for each $J$ on $\Omeg{J}:=\{\w{\varnothing}{J},\w{J}{J}\}$ with
\begin{equation}\label{eq: omega_leaves}
\X{t}{J}=\begin{cases} \w{\varnothing}{J}, & \text{with probability } 1-\rh{J},\\
\w{J}{J}, & \text{with probability } \rh{J}.
\end{cases}
\end{equation}
If $J=\varnothing$, then obviously $\X{t}{J}=\w{\varnothing}{\varnothing}$ with probability 1 (for consistency, set $\rh{\varnothing}:=0$).  Now consider the internal fragments $\I{\alpha}$, $\alpha \in G$. As already mentioned, every fragment may be pieced together from its two descendant fragments
$\Ileq{\alpha}=\{\beta\in \I{\alpha}: \beta<\alpha\}$ and $\Igeq{\alpha}=\{\beta\in \I{\alpha}: \beta>\alpha\}$. Namely,  $\I{\alpha}=\Ileq{\alpha}\,\cup\,\{\alpha\}\,\cup\,\Igeq{\alpha}$;  $\Ileq{\alpha}$ and  $\Igeq{\alpha}$ may be internal fragments or (empty or full) external fragments. We now proceed down the tree inductively by taking, in every step, one $\alpha$, for which $\X{t}{\Ileq{\alpha}}$ and $\X{t}{\Igeq{\alpha}}$  have already been defined (as independent processes on $\Omeg{\Ileq{\alpha}}$ and $\Omeg{\Igeq{\alpha}}$).  $\X{t}{\I{\alpha}}$ will live on the state space
$
\Omeg{\I{\alpha}}:=\Big\{\w{\varnothing}{\I{\alpha}},\w{\alpha}{\I{\alpha}},\wind{\I{\alpha}},\wdep{\I{\alpha}}\Big \}, 
$
and we define the composite event
$
\w{\I{\alpha}}{\I{\alpha}}:=\Omeg{\I{\alpha}} \setminus \w{\varnothing}{\I{\alpha}},
$
which blends in with the notation in \eqref{eq: omega_leaves}. By slight abuse of notation, we will sometimes write $\X{t}{J}=\w{J}{J}$ for a $J\in \S$ even though the correct statement would be $\X{t}{J}\in\w{J}{J}$ if $J$ is an internal fragment, and  $\X{t}{J}=\w{J}{J}$ if $J$ is an external fragment.
We now construct $\X{t}{\I{\alpha}}$  by specifying
\begin{equation*} \label{eq:law_Xt_1}
\X{t}{\I{\alpha}} = \wind{\I{\alpha}} \quad \text{if } \X{t}{\Ileq{\alpha}}=\w{\Ileq{\alpha}}{\Ileq{\alpha}} \quad \text{or}\quad  \X{t}{\Igeq{\alpha}}=\w{\Igeq{\alpha}}{\Igeq{\alpha}}
\end{equation*}
(this is the case with probability $1-\lam{\alpha}{\I{\alpha}}$) and, if $ \X{t}{\Ileq{\alpha}}=\w{\varnothing}{\Ileq{\alpha}} \text{ and }  \X{t}{\Igeq{\alpha}}=\w{\varnothing}{\Igeq{\alpha}}$ (which is the case with probability $\lam{\alpha}{\I{\alpha}}$), we set 
\begin{equation*} \label{eq:law_Xt_2}
\X{t}{\I{\alpha}} =\begin{cases}
\w{\varnothing}{\I{\alpha}},& \text{with probability } \nicefrac{(1-\rh{\I{\alpha}})\,}{\,\lam{\alpha}{\I{\alpha}}},\\
\w{\alpha}{\I{\alpha}}, &\text{with probability } \nicefrac{\rh{\alpha}\,}{\,\lam{\alpha}{\I{\alpha}}},\\
\wdep{\I{\alpha}}, & \text{with probability } \nicefrac{\rh{\Ileq{\alpha}}\ \cdot\ \rh{\Igeq{\alpha}}\,}{\,\lam{\alpha}{\I{\alpha}}}
 \end{cases}
\end{equation*}
independently of what has been decided for the previous fragments. Here, $\lam{\alpha}{\I{\alpha}}$ is defined as in~\eqref{def:lambda}, but with respect to the link set $I_\alpha$ of the tree with root $\alpha$. This means that $\Omeg{\Ileq{\alpha}}\times\Omeg{\Igeq{\alpha}}\setminus (\w{\varnothing}{\Ileq{\alpha}},\w{\varnothing}{\Igeq{\alpha}})$ is identified with the event $\wind{\I{\alpha}}$, whereas the   remaining element $(\w{\varnothing}{\Ileq{\alpha}},\w{\varnothing}{\Igeq{\alpha}})$ is `split up' into the elements of $\Omeg{\I{\alpha}}\setminus \wind{\I{\alpha}}$. 

We see that under this construction, $\X{t}{\I{\alpha}}$ has the law 
\begin{equation}\label{eq:law_subprocesses_Ialpha}
\begin{split}
\X{t}{\I{\alpha}} =\begin{cases}
\w{\varnothing}{\I{\alpha}},& \text{with probability } 1-\rh{{\I{\alpha}}},\\
\w{\alpha}{\I{\alpha}},&  \text{with probability } \rh{\alpha},\\
\wdep{\I{\alpha}},&  \text{with probability }\rh{\Ileq{\alpha}}\cdot\,\rh{\Igeq{\alpha}},\\
\wind{\I{\alpha}}& \text{with probability } 1-\lam{\alpha}{\I{\alpha}}.
\end{cases}
\end{split}
\end{equation}
The construction is completed when $\X{t}{L}=\X{t}{\I{\gamma}}$ has been reached. Altogether, we then have the family $\X{t}{}=(\X{t}{J})^{}_{J \in \S{}}$ with state space $\Omega:=\bigtimes_{J\in\S{}} \Omeg{J}$. The sequence of random variables $X=(\X{t}{})^{}_{t\in \NN}$ is defined to be i.i.d. in $t$. The $\X{t}{J}$ are independent for all disjoint fragments; in particular, for every stump set $R\in\Rr{\Ts}$, the family $(\X{t}{J})^{}_{J\in\L{R}{}}$ with state space $\bigtimes_{J\in\L{R}{}} \Omeg{J}$ is independent. In contrast, for nondisjoint fragments there are dependencies, such as
\begin{equation}\label{eq:omega_varnothing}
\X{t}{\I{\alpha}} \in \Omeg{\I{\alpha}}\setminus \wind{\I{\alpha}} \text{ for } \alpha\in G \text{ implies } \X{t}{J} = \w{\varnothing}{J} \text{ for all } J\in \S^{\I{\alpha}}\setminus \I{\alpha}.
\end{equation}
The other way round, this means
\begin{equation}\label{eq:omega_I}
\X{t}{J} \neq \w{\varnothing}{J} \text{ for  some }  J \in \S^{\I{\alpha}}\setminus\I{\alpha} \text{ implies } \X{t}{\I{\alpha}} \in \w{\I{\alpha}}{\I{\alpha}}.
\end{equation}

\paragraph{Events and Waiting times.} We now define events $\Ex(s)$ for all $s\in \Gamma$ based on the process $\X{t}{}$. To this end, define $\pi^{}_I:\Omega \to \Omeg{I}$, $I\in\S$, as the canonical projection. We set for all $\alpha\in G$ and $J\in \S{}$:  
\[
\Ex(\alpha):=\Big\{\omega\in \Omega: \pi^{}_{\I{\alpha}}(\omega)=\w{\alpha}{\I{\alpha}}\Big\} ,\qquad \Ex(J):=\Big\{\omega\in \Omega :\pi^{}_J(\omega)=\w{J}{J}\Big\}.
\]
Due to \eqref{eq:omega_varnothing} and \eqref{eq:omega_I}, these events satisfy the nesting condition \eqref{cond:nested_events}. 

\begin{example}\label{ex:aux_events}
Consider the fragmentation tree in Figure~\ref{fig:tree_exam}. For every $t\in\NN$, $\X{t}{}$ is given by the family $
\X{t}{}=\Big(\X{t}{\Ileq{3}},\X{t}{\Ileq{4}},\X{t}{\Igeq{4}},\X{t}{\I{4}},\X{t}{\I{3}}\Big).
$
Events $\omega\in\Omega$ that satisfy $\P(\X{t}{}=\omega)>0$ are   
\[
\begin{array}{lll}
\Big(\w{\varnothing}{\Ileq{3}},\w{\varnothing}{\Ileq{4}},\w{\varnothing}{\Igeq{4}},\w{\varnothing}{4},\w{\varnothing}{\I{3}}\Big),
& \Big(\w{\varnothing}{\Ileq{3}},\w{\varnothing}{\Ileq{4}},\w{\varnothing}{\Igeq{4}},\w{\varnothing}{\I{4}},\w{3}{\I{3}}\Big) ,
& \Big(\w{\varnothing}{\Ileq{3}},\w{\varnothing}{\Ileq{4}},\w{\varnothing}{\Igeq{4}},\w{\varnothing}{\I{4}},\wdep{\I{3}}\Big),\\[.5em]
\Big(\w{\varnothing}{\Ileq{3}},\w{\varnothing}{\Ileq{4}},\w{\varnothing}{\Igeq{4}},\w{4}{\I{4}},\wind{\I{3}}\Big),
& \Big(\w{\varnothing}{\Ileq{3}},\w{\varnothing}{\Ileq{4}},\w{\varnothing}{\Igeq{4}},\wdep{\I{4}},\wind{\I{3}}\Big), 
& \Big(\w{\Ileq{3}}{\Ileq{3}},\w{\varnothing}{\Ileq{4}},\w{\varnothing}{\Igeq{4}},\w{\varnothing}{\I{4}},\wind{\I{3}}\Big),\\[.5em]
\Big(\w{\Ileq{3}}{\Ileq{3}},\w{\varnothing}{\Ileq{4}},\w{\varnothing}{\Igeq{4}},\w{4}{\I{4}},\wind{\I{3}}\Big),
& \Big(\w{\Ileq{3}}{\Ileq{3}},\w{\varnothing}{\Ileq{4}},\w{\varnothing}{\Igeq{4}},\wdep{\I{4}},\wind{\I{3}}\Big),
& \Big(\w{\varnothing}{\Ileq{3}},\w{\varnothing}{\Ileq{4}},\w{\Igeq{4}}{\Igeq{4}},\wind{\I{4}},\wind{\I{3}}\Big),\\[.5em]
\Big(\w{\Ileq{3}}{\Ileq{3}},\w{\varnothing}{\Ileq{4}},\w{\Igeq{4}}{\Igeq{4}},\wind{\I{4}},\wind{\I{3}}\Big). & & 
\end{array}
\]
Events of interest are, for example,
\[
\Ex(4)=\Big\{ \Big(\w{\varnothing}{\Ileq{3}},\w{\varnothing}{\Ileq{4}},\w{\varnothing}{\Igeq{4}},\w{4}{\I{4}},\wind{\I{3}}\Big),\ \Big(\w{\Ileq{3}}{\Ileq{3}},\w{\varnothing}{\Ileq{4}},\w{\varnothing}{\Igeq{4}},\w{4}{\I{4}},\wind{\I{3}}\Big)\Big\}
\]
and 
\begin{align*}
\Ex(\I{4})\setminus\Ex\Big(\L{\G{4}{\varnothing}}{\I{4}}\Big)=\Big\{ & \Big(\w{\varnothing}{\Ileq{3}},\w{\varnothing}{\Ileq{4}},\w{\varnothing}{\Igeq{4}},\w{4}{\I{4}},\wind{\I{3}}\Big), \Big(\w{\varnothing}{\Ileq{3}},\w{\varnothing}{\Ileq{4}},\w{\varnothing}{\Igeq{4}},\wdep{\I{4}},\wind{\I{3}}\Big),\\[.5em]
& \Big(\w{\Ileq{3}}{\Ileq{4}},\w{\varnothing}{\Ileq{4}},\w{\varnothing}{\Igeq{4}},\w{4}{\I{4}},\wind{\I{3}}\Big), \Big(\w{\Ileq{4}}{\Ileq{3}},\w{\varnothing}{\Ileq{4}},\w{\varnothing}{\Igeq{4}},\wdep{\I{4}},\wind{\I{3}}\Big)\Big\}. \qedhere
\end{align*}
\end{example}

\begin{SCfigure}[50]
\setlength{\fboxsep}{5pt}
\includegraphics[width=.45\textwidth]{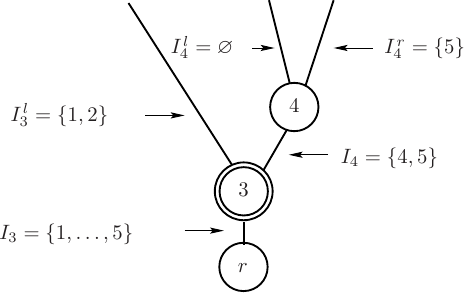}
\caption{\label{fig:tree_exam}Fragmentation tree with vertex set $G=\{3,4\}$, link set $ L=\{1,\ldots,5\}$, internal fragments $\I{3}$ and $\I{4}$, and external fragments $\Ileq{3}$, $\Ileq{4}$ and $\Igeq{4}$. Here, $\L{\G{4}{\varnothing}}{\I{4}}=\big\{\Ileq{4},\Igeq{4}\big\}$.}
\end{SCfigure}

Let $\t{\Ex(s)}$ denote the waiting time for the event $\Ex(s)$, $s \in \Gamma$ (condition \eqref{cond:sigma} is then obviously satisfied).
By construction, $\t{\Ex(\alpha)}$ and $\t{\Ex(J)}$ are geometrically distributed with parameters $\rh{\alpha}$ and $\rh{J}$, $\alpha \in G$, $J \in \S$. 
Since for every $\alpha\in G, H\subseteq E$, the family $\big(\X{t}{J}\big)^{}_{J\in \L{\G{\alpha}{H}}{\I{\alpha}}}$  is independent, the family of waiting times $(\t{\Ex(J)})_{J\in \L{\G{\alpha}{H}}{\I{\alpha}}}$ is independent as well, and, as a minimum of independent geometric variables, $\t{\Ex(\L{\G{\alpha}{H}}{\I{\alpha}})}$ is geometrically distributed with parameter $1-\lam{\G{\alpha}{H}}{\I{\alpha}}$. 
For any $H\subseteq E$, the waiting time $\t{\Ex(\I{\alpha})\setminus \Ex(\L{\G{\alpha}{H}}{\I{\alpha}})}$  is geometric with parameter $\rh{\I{\alpha}}-(1-\lam{\G{\alpha}{H}}{\I{\alpha}})=\lam{\G{\alpha}{H}}{\I{\alpha}}-\lam{\varnothing}{\I{\alpha}}$ (recall that $\Ex(\L{\G{\alpha}{H}}{\I{\alpha}})\subseteq \Ex(\I{\alpha})$ by Fact~\ref{fact:nested_events_1}).
Since the conditions of Proposition~\ref{thm:Max_in_Min} are satisfied, we can directly conclude 

\begin{corol}\label{coro:Max_Min_X_prozess}
Let $\Ts=(\gamma,G,E,L)$ be a fragmentation tree. Then
\[
\P\big(\Max{\Ex}{\varnothing}\cap \m{\Ex}{\varnothing}\big)= \sum\limits_{H \subseteq E} (-1)_{}^{|H|}~ \P \big(\Min{\Ex}{H}\cap \m{\Ex}{H}\big),
\]
with $\m{\Ex}$, $\Min{\Ex}$ and $\Max{\Ex}$ as in \eqref{def:Max_G}--\eqref{def:Min_G(t)}, and $\Ev$ replaced by $\Ex$.
\end{corol}

\subsubsection{Constructing the fragmentation process from the auxiliary process.}  We now present a pathwise construction for realisations of $(\F{t}{})_{t \in\NN}^{}$ that have the correct law as long as they are compatible with a given fragmentation tree $\Ts=(\gamma,G,E,L)$. We say that $\F{t}{}$ is \emph{compatible} with $\Ts$ if $\F{t}{}\in\Rr{\Ts}$. In this case, $(\F{t'}{})^{}_{0\leqslant t'\leqslant t}$ matches a stump tree of $\Ts$. We use the auxiliary process $(\X{t}{})_{t\in\NN}^{}$ for the construction.

Recall that the transition from $\F{t-1}{}$ to $\F{t}{}$ is determined by the family of independent random variables $(\Af{t}{J})_{J\in \mathcal L_{ F_{t-1}^{}}}$ (see Definition~\ref{def:Fragmentation_discrete}). Now fix a tree $\Ts=(\gamma,G,E,L)$ and construct the enlarged family $\big(\Af{t}{J}\big)^{}_{J\in\S{}}$ from $\X{t}{}$ by prescribing that, for all $t> 0$, 
\begin{align}\label{eq:coupling} 
\begin{split}
\Af{t}{J} = \varnothing, & \text{ if and only if } \X{t}{J}=\w{\varnothing}{J}, \text{ for all } J\in\S{},\\
\Af{t}{\I{\alpha}} = \{\alpha\}, & \text{ if and only if } \X{t}{\I{\alpha}}=\w{\alpha}{\I{\alpha}}, \text{ for all } \alpha\in G.
\end{split}
\end{align} 
This entails that $\Af{t}{J}=\varnothing$ with probability $1-\rh{J}$ for all $J\in\S$, and $\Af{t}{\I{\alpha}}=\{\alpha\}$ with probability $\rh{\alpha}$, $\alpha\in G$. On the other hand, it implies that
\begin{equation}\label{eq:law_A1}
\Af{t}{J} \in J,\  \text{ if and only if }  \X{t}{J}=\w{J}{J}, \text{ for } J\in\L{G}{},  
\end{equation}
which happens with probability $\rh{J}$,
and 
\begin{equation}\label{eq:law_A2}
\Af{t}{\I{\alpha}} \in \I{\alpha} \setminus \alpha,\ \text{ if and only if } \X{t}{\I{\alpha}}\in\w{\I{\alpha}}{\I{\alpha}}\setminus \w{\alpha}{\I{\alpha}}, \text{ for } \alpha\in G,
\end{equation}
which is the case with probability $\rh{\I{\alpha}\setminus\, \alpha}$. If we want to know the precise event in these cases, we can use additional randomness to decide for $\Af{t}{J}=\{\beta\}$ with probability $\rh{\beta}$ for all $\beta\in J\in\L{G}{}$, and $\Af{t}{\I{\alpha}}=\{\beta\}$, $\beta\in\I{\alpha}\setminus\alpha$, for all $\alpha\in G$, but this is never required in our construction; what matters is that, under the construction in \eqref{eq:coupling}, each $\Af{t}{J}$ has the right probabilities for the \emph{compatible events} (those in \eqref{eq:coupling}) and their complements, for every given $t$ and every given $J\in \S{}$. Also, the ${\Af{t}{J}}$ inherit from the $\X{t}{J}$ the i.i.d. property over $t$ and the independence across disjoint fragments. 

We now proceed as follows. Start with $\F{0}{}=\varnothing$, which is certainly compatible with the given $\Ts$. If $\F{t-1}{}$ is compatible,  then construct $\F{t}{}$ from $\F{t-1}{}$ according to  Definition~\ref{def:Fragmentation_discrete}, but use
the $\big(\Af{t}{J}\big)^{}_{J\in\S{}}$ from \eqref{eq:coupling}--\eqref{eq:law_A2}. If only \emph{compatible events} occur for all $J \in \L{ F_{t-1}^{}}{}$, then  $\F{t}{}$ is compatible as well. If at least one \textit{incompatible} event occurs (at least one event of those in \eqref{eq:law_A1} or \eqref{eq:law_A2}), then $\F{t}{}$ is incompatible. We say the construction \emph{fails} at time $t$ and discontinue it.
Since the subfamily $(\Af{t}{J})^{}_{J\in\mathcal L_{F_{t-1}^{}}}$ has the right law for the compatible events, we know that $(\F{t}{})_{t\in \NN}^{}$ has the right law for all $t< t_f$, where $t_f$ is the failure time.

\begin{proposition} \label{prop:taus_in_Ts}
For every given fragmentation tree $\Ts=(\gamma, G,E,L)$ and the pathwise construction of $\F{}{}$ described above, we have
\begin{equation}\label{eq:events_F_X}
\Max{\Ef}{\varnothing}\cap  \m{\Ef}{\varnothing} = \Max{\Ex}{\varnothing}\cap  \m{\Ex}{\varnothing}, \quad t\in \mathbb N_0
\end{equation}
and
\begin{equation}\label{eq:law_F_in_X}
\P\Big(\Max{\Ef}{\varnothing}\cap \m{\Ef}{\varnothing}\Big) = \P\Big( \Max{\Ex}{\varnothing}\cap  \m{\Ex}{\varnothing}\Big), \quad t\in \mathbb N_0,
\end{equation}
where $\Max{\Ef}$ and $\mh{\Ef}$ are as in \eqref{def:Max_G} and \eqref{def:m_G} with $\Ev$ replaced by $\Ef$.
\end{proposition}

The description in terms of the waiting times of the auxiliary process offers a great advantage since this law is known and does not change over time. 

\begin{proof}
We start by considering the events $\Max{\Ev}{\varnothing}$ and $\m{\Ev}{\varnothing}$ for general $\Ev$. We know by definition that $\Max{\Ev}{\varnothing}\subseteq\{\t{\Ev(\alpha)}<\t{\Ev(\L{G}{})} \text{ for all } \alpha\in G\}$. Since by Fact~\ref{fact:nested_events}~\ref{fact:nested_events_3} furthermore $\t{\Ev(\L{G}{})}\leqslant\t{\Ev(\L{\G{\alpha}{\varnothing}}{\I{\alpha}})}$ for all $\alpha\in G$, we obtain
\begin{align*}\label{eq:ancest_rel}
\Max{\Ev}{\varnothing}& \cap \Big\{\t{\Ev(\alpha)} = \t{\Ev(\I{\alpha})\setminus \Ev(\L{\G{\alpha}{\varnothing}}{\I{\alpha}})}\Big\} \\
& =\Max{\Ev}{\varnothing}\cap \Big\{\t{\Ev(\alpha)}=\min\Big\{\t{\!\Ev(\I{\alpha})\setminus \Ev(\L{\G{\alpha}{\varnothing}}{\I{\alpha}})},\t{\Ev(\L{\G{\alpha}{\varnothing}}{\I{\alpha}})}\Big\}\Big\}\\
& = \Max{\Ev}{\varnothing}\cap \Big\{\t{\Ev(\alpha)}=\t{\Ev(\I{\alpha})}\Big\}
\end{align*}
for every $\alpha\in G$. We can therefore rewrite
\begin{equation}\label{eq:unreduced_1}
\Max{\Ev}{\varnothing}\cap \m{\Ev}{\varnothing}=\Max{\Ev}{\varnothing}\cap  \bigcap_{\alpha\in G} \big\{\t{\Ev(\alpha)}=\t{\Ev(\I{\alpha})}\big\}.
\end{equation}
The choice  $\Ev=\Ef$ or $\Ev=\Ex$ in \eqref{eq:unreduced_1} turns the claim \eqref{eq:events_F_X} into
\begin{equation}\label{eq:unreduced_2}
\Max{\Ef}{\varnothing}\cap  \bigcap_{\alpha\in G} \!\big\{\t{\Ef(\alpha)}=\t{\Ef(\I{\alpha})}\big\}
=\Max{\Ex}{\varnothing}\cap \bigcap_{\alpha\in G} \!\big\{\t{\Ex(\alpha)}=\t{\Ex(\I{\alpha})}\big\}.
\end{equation}
Recall that $\Ef(s)=\{s\}$ for all $s\in \Gamma$, such that $\t{\Ef(s)}=\t{ s}=\min\big \{\t{\alpha}:\alpha\in s\big \}$ is the time at which the first link in $s$ is removed. Now, assume that we have shown the identification
\begin{equation}\label{eq:equivalence_occurence_times}
\t{J}  = \t{\Ex(J)} \text{ for all } J\in  \S{ } \text{ given } \bigcap_{\alpha\in G} \{\t{\alpha}=\t{\I{\alpha}}\}.
\end{equation}
Due to \eqref{eq:coupling}--\eqref{eq:law_A2}, it then follows under the pathwise construction of $\F{}{}$ from the auxiliary process that 
$
\big\{\t{\alpha}=\t{\I{\alpha}}\}\,=\,\big\{\t{\Ex(\alpha)}=\t{\Ex(\I{\alpha})}\}
$. 
Together with \eqref{eq:equivalence_occurence_times}, this implies $\t{\alpha}=\t{\Ex(\alpha)}$ for all $\alpha\in G$. Equation \eqref{eq:equivalence_occurence_times} therefore entails \eqref{eq:unreduced_2}, so
it suffices to show \eqref{eq:equivalence_occurence_times}.

We first show the relation \eqref{eq:equivalence_occurence_times} for all internal fragments (i.e. for all $\I{\alpha}$, $\alpha\in G$).
Start with the set of links $\I{\gamma}=L$ and initial value $\F{0}{L}=\{\varnothing\}$. For $t\geqslant 1$, the first event $\big\{\F{t}{L}\neq\varnothing\big \}$ happens when $\{\Af{t}{L}\in L\}$ for the first time; this happens at $t=\t{L}$. Under \eqref{eq:coupling}, $\t{L}$ corresponds to the first time at which $\{\X{t}{L} \in \w{L}{L}\}$ (at time $\t{\Ex(L)}$); this gives $\t{L}=\t{\Ex(L)}$.

Now consider a link $\beta\in G\setminus \{\gamma\}$, and assume that we have already identified $\t{\I{\nu}}=\t{\Ex(\I{\nu})}$ for the parent node $\nu$ of $\beta$. Given $\t{\nu}=\t{\I{\nu}}$, we conclude $\t{\I{\nu}}<\t{\I{\beta}}$ since $\nu\notin\I{\beta}\subset \I{\nu}$. This yields $\F{{\mathcal T}^{}_{\I{\nu}}}{\I{\beta}}=\varnothing$ by \eqref{eq:recursion_F}. 
Now consider the first time $t'>\t{\I{\nu}}$ the event $\big \{\F{t'}{\I{\beta}}\neq\varnothing\big \}$ occurs. Again by \eqref{eq:recursion_F}, this time is $\t{\I{\beta}}$. Due to \eqref{eq:coupling}, the event $\big \{\F{t'}{\I{\beta}}\neq\varnothing\big \}$ with $t'>\t{\I{\nu}}$ happens when $\big \{\X{t'}{\I{\beta}}\in\w{\I{\beta}}{\I{\beta}}\big \}$ for the first time, which is at time $\t{\Ex(\I{\beta})}$. Since we assumed that $\t{\I{\nu}}=\t{\Ex(\I{\nu})}$ and since $\t{\Ex(\I{\nu})}\leqslant\t{\Ex(\I{\beta})}$ due to \eqref{cond:nested_events}, we conclude $\t{\Ex(\I{\beta})}>\t{\I{\nu}}$, which gives $\t{\Ex(\I{\beta})}=\t{\I{\beta}}$.

It remains to show the equality of the waiting times for the full external fragments $J\in\L{G}{}$. For each such fragment $J$, denote by $\delta:=\delta_J\in G$ the unique link for which $J\in \{\Ileq{\delta},\Igeq{\delta}\}$. Assume that $\t{\delta}=\t{\I{\delta}}$ and one has already identified $\t{\I{\delta}}=\t{\Ex(\I{\delta})}$. 
With the same arguments as above, we conclude that under the given assumption $\t{J}=\t{\Ex(J)}$. 

To finally show \eqref{eq:law_F_in_X} recall that, for a given fragmentation tree $\Ts$, $(\F{t}{})_{t\in\NN}^{}$ has the right law for all $t<t_f$, where $t_f$ is the first time at which $\F{t_f}{}$ fails to be compatible with the tree. Since $\Max{\Ef}{\varnothing}\cap \m{\Ef}{\varnothing}$ describes a sequence of events that are all compatible with $\Ts$, \eqref{eq:law_F_in_X} follows.
\end{proof}

\subsection{Explicit tree probabilities}\label{sec:expl_tree_prob_discrete}
We can now harvest the consequences and state an explicit expression for tree probabilities. 

\begin{theorem} \label{prop_discrete}
For a given fragmentation tree $\Ts=(\gamma, G,E,L)$, one has $\P\big(\Ftree\big)=(1-\rh{L})^t=\big (\lam{\varnothing}{L}\big)^t$ for $G=\varnothing$, and, for $G\neq\varnothing$, 
\begin{equation*} \label{eq:prop_tree_discrete}
\P\big(\Ftree\big)=\sum\limits_{H\subseteq E}(-1)^{|H|}~\Big[\big(\lam{\G{\gamma}{H}}{L}\big)^t-\big(\lam{\varnothing}{L}\big)^t\Big]~\prod\limits_{\alpha\in G}~\frac{\rh{\alpha}}{\lam{\G{\alpha}{H}}{\I{\alpha}}-\lam{\varnothing}{\I{\alpha}}},
\end{equation*}
where the $\lambda$'s are defined as in~\eqref{def:lambda}.
\end{theorem}

\begin{proof}
We first employ Proposition~\ref{prop:taus_in_Ts} together with Corollary~\ref{coro:Max_Min_X_prozess} to rewrite the matching probability corresponding to the fragmentation process in terms of the auxiliary process:
\begin{equation}\label{eq:Min_Max_auxiliary}
\P\big(\Ftree\big)=\P\big(\Max{\Ex}{\varnothing}\cap  \m{\Ex}{\varnothing}\big)= \sum\limits_{H\subseteq E} (-1)^{|H|}\, \P\big( \Min{\Ex}{H}\cap \m{\Ex}{H}\big).
\end{equation}
Now fix a set of edges $H\subseteq E$ and consider the event $\Min{\Ex}{H}$ on the right-hand side of \eqref{eq:Min_Max_auxiliary}. The family $\big(\t{\Ex(\alpha)}\big)^{}_{\alpha\in \G{\gamma}{H}}$ is not independent, so  $\min\{\t{\Ex(\alpha)} : \alpha\in \G{\gamma}{H}\}$ is not a simple  geometric waiting time. But, taking the intersection with $\m{\Ex}{H}$, we can use that $\min\{\t{\Ex(\alpha)}:\alpha\in \G{\gamma}{H}\}=\t{\Ex(\gamma)}$ by Fact~\ref{fact:nested_events}~\ref{fact:nested_events_2} and, again intersecting with $\m{\Ex}{H}$, that $\t{\Ex(\gamma)}=\t{\Ex(\L{\varnothing}{})\setminus \Ex(\L{\G{\gamma}{H}}{})}$ since $\I{\gamma}=L=\L{\varnothing}{}$. This gives
\begin{align*}\label{eq:rewrite_minb}
\Min{\Ex}{H}\cap \m{\Ex}{H} &= \big\{\t{\Ex(\L{\varnothing}{}) \setminus \Ex(\L{\G{\gamma}{H}}{})}\leqslant t,\, t<\t{\Ex(\L{\G{\gamma}{H}}{})}\big\}\cap \m{\Ex}{H}\\
& =\big\{\min\big\{\t{\Ex(\L{\varnothing}{})\setminus \Ex(\L{\G{\gamma}{H}}{})},\t{\Ex(\L{\G{\gamma}{H}}{})}\big\}\! \leqslant t,\, t<\t{\Ex(\L{\G{\gamma}{H}}{})}\big\}\\
& \hspace{1cm} \cap \m{\Ex}{H}\\
& = \big\{\t{\Ex(\L{\varnothing}{})}\leqslant t,\, t < \t{\Ex(\L{\G{\gamma}{H}}{})}\big\}\cap \m{\Ex}{H}.
\end{align*} 
Let us now investigate the connection between $\big \{\t{\Ex(\L{\varnothing}{})}\leqslant t,\,t < \t{\Ex(\L{\G{\gamma}{H}}{})} \big \}$ and $\m{\Ex}{H}$.
To this end, consider first an $\alpha\notin \G{\gamma}{H}$. For this we know that there is a $J\in\L{\G{\gamma}{H}}{}$ such that $\I{\alpha}\subseteq J$; thus $\Ex(\alpha)\subseteq \Ex(\I{\alpha})\setminus \Ex(\L{\G{\alpha}{H}}{\I{\alpha}})\subseteq \Ex(\L{\G{\gamma}{H}}{})\subseteq \Ex(\L{\varnothing}{})$ by \eqref{cond:nested_events}.
Since the minimum of a collection of events is independent of the order in which (some of) the events occur, we obtain the independence of $\big\{\t{\Ex(\L{\varnothing}{})}\leqslant t,\,t< \t{\Ex(\L{\G{\gamma}{H}}{})}\big\}$ and $\big \{\t{\Ex(\alpha)}=\t{\Ex(\I{\alpha})\setminus  \Ex(\L{\G{\alpha}{H}}{\I{\alpha}})}\big\}$ for every $\alpha\notin\G{\gamma}{H}$.
Consider now $\alpha\in \G{\gamma}{H}$. We can then obviously decompose the event $\big \{\t{\Ex(\L{\varnothing}{})}\leqslant t,\,t < \t{\Ex(\L{\G{\gamma}{H}}{})}\big \}$ into 
\[
\big\{\t{\Ex(\L{\varnothing}{})}\leqslant t\big\} \cap \big\{\t{\Ex(\L{\G{\gamma}{H}}{L}\setminus \L{\G{\alpha}{H}}{\I{\alpha}})}> t\big\}\cap \big\{\t{\Ex(\L{\G{\alpha}{H}}{\I{\alpha}})}>t\big\}.
\]
Due to the independence of the $\X{t}{J}$'s for disjoint sets $J$, we conclude that the event $\big \{\t{\Ex(\alpha)}=\t{\Ex(\I{\alpha})\setminus  \Ex(\L{\G{\alpha}{H}}{\I{\alpha}})} \big \}$ is independent of $\big \{\t{\Ex(\L{\G{\gamma}{H}}{}\setminus \L{\G{\alpha}{H}}{\I{\alpha}})}>t\big \}$. The independence of $\big \{\t{\Ex(\alpha)}=\t{\Ex(\I{\alpha})\setminus  \Ex(\L{\G{\alpha}{H}}{\I{\alpha}})}\big \}$ and $\big \{\t{\Ex(\L{\G{\alpha}{H}}{\I{\alpha}})}>t\big \}$ is obvious since the respective events $\Ex(\I{\alpha})\setminus  \Ex(\L{\G{\alpha}{H}}{\I{\alpha}})$ and $\Ex(\L{\G{\alpha}{H}}{\I{\alpha}})$ are disjoint; the independence of $\{\t{\Ex(\L{\varnothing}{})}\leqslant t\}$ follows again by the argument that the minimum of a collection of events is independent of the order in which (some of) the events occur. Altogether, we obtain
\begin{align*}
\P \big(\Min{\Ex}{H}\cap \m{\Ex}{H}\big) &  = \big [ \P \big(\t{\Ex(\L{\G{\gamma}{H}}{})}>t \big) - \P\big(\t{\Ex(\L{\varnothing}{}{})} >t\big) \big ]\times  \P\big(\m{\Ex}{H}\big),
\end{align*}
where we used that $\Ex(\L{\G{\gamma}{H}}{})\subseteq \Ex(\L{\varnothing}{})$ by Fact~\ref{fact:nested_events}~\ref{fact:nested_events_1}.
Since for $\alpha,\beta\in G$ with $\alpha\prec \beta$,  $\Ex(\alpha)\notin \Ex(\I{\beta})$ and hence $\Ex(\alpha)\notin \Ex(\I{\beta})\setminus \Ex(\L{\G{\alpha}{H}}{})$, we can furthermore decompose the probability for $\m{\Ex}{H}$ into independent factors:
\begin{align*}
\P\big(\m{\Ex}{H}\big)& = \prod\limits_{\alpha \in G} \P\big(\t{\Ex(\alpha)}= \t{\Ex(\I{\alpha}) \setminus \Ex(\L{\G{\alpha}{H}}{\I{\alpha}})}\big).
\end{align*}
Now recall that each $\t{\Ex(\L{\G{\alpha}{H}}{\I{\alpha}})}$ is geometric with parameter $1-\lam{\G{\alpha}{H}}{\I{\alpha}}$ and that each $\t{\Ex(\I{\alpha}) \setminus \Ex(\L{\G{\alpha}{H}}{\I{\alpha}})}$ is geometric with parameter $\lam{\G{\alpha}{H}}{\I{\alpha}}-\lam{\varnothing}{\I{\alpha}}$. All in all, we obtain
\begin{align*}
\P \big(\Min{\Ex}{H}\cap \m{\Ex}{H}\big) 
& =\Big [\big(1-\big (1-\lam{\G{\gamma}{H}}{L} \big )\big)^t\!-\big (1-\big (1-\lam{\varnothing}{L}\big) \big)^t \Big ] \prod\limits_{\alpha \in G} \frac{\rh{\alpha}}{\lam{\G{\alpha}{H}}{\I{\alpha}}-\lam{\varnothing}{\I{\alpha}}} \\
& =\Big[ \big(\lam{\G{\gamma}{H}}{L} \big)^t-\big(\lam{\varnothing}{L}\big)^t\Big]\ \prod\limits_{\alpha \in G}\  \frac{\rh{\alpha}}{\lam{\G{\alpha}{H}}{\I{\alpha}}-\lam{\varnothing}{\I{\alpha}}}\ .
\end{align*}
Equation~\eqref{eq:Min_Max_auxiliary} then completes the proof.
\end{proof}

\begin{corol}\label{coro:sum_trees_discrete}
For every  $G\subseteq L$, the probability of the fragmentation process to be in state $G$ at time $t\in \NN$ is given by
\begin{equation}\label{eq:sum_trees_discrete}
\P\big(\F{t}{}=G\big)= \sum_{\Ts\in\tau(G,L)} \P\big(\Ftree\big), 
\end{equation}
with $\P\big(\Ftree\big)$ as in Theorem~\ref{prop_discrete} and $\tau(G,L)$  the set of all fragmentation trees with vertex set $G$ and underlying link set $L$. 
\end{corol}

There is (in general) no simple explicit expression for the sum in \eqref{eq:sum_trees_discrete}. But the following reformulation of the $\lambda$'s 
\[
\lam{\G{\alpha}{H}}{\I{\alpha}}-\lam{\varnothing}{\I{\alpha}}=\prod_{J\in\L{\G{\alpha}{H}}{\I{\alpha}}}\big(1-\rh{J}\big)-\Big(1-\sum_{\nu\in \I{\alpha}} \rh{\nu} \Big)= \sum_{\nu \in \G{\alpha}{H}} \rh{\nu} + \sum_{\substack{J \subseteq \L{\G{\alpha}{H}}{\I{\alpha}},\\[.1em]
|J|>1}} (-1)_{}^{|J|} \prod_{I\in J}\, \rh{I}
\]
shows that there is one exception, namely the case $|\L{\G{\alpha}{H}}{\I{\alpha}}|\leqslant 1$ for every $\alpha\in G$ and every $H\subseteq E$. If $L=\{1,\ldots,n\}$, this is true for $G\subseteq \{1,n\}$, in which case 
\begin{equation*}\label{eq:sum_discrete}
\P\big(\F{t}{}=G\big)=\sum_{\varnothing\neq H\subseteq G} (-1)_{}^{|H|}\ \big[ \big(\lam{H}{L} \big)^t-\big(\lam{\varnothing}{L}\big)^t\big].
\end{equation*}
This goes together with observations in \cite{WangenheimBaakeBaake}, where a subset of links that only contains the `ends' of $L$ induced significant simplifications. This now becomes clear in the light of our event structure: $\alpha\in \{1,n \}$ implies that either $\Ileq{\alpha}=\varnothing$ or $\Igeq{\alpha}=\varnothing$, so that the probability for $\wdep{\I{\alpha}}$ vanishes, cf. \eqref{eq:law_subprocesses_Ialpha}.

\subsection{Continuous time}\label{sec:tree_prob_cont}
In discrete time, removing a given link forbids to remove any other link in the same fragment in the same time step. This is different in the analogue process in continuous time, which was treated comprehensively in \cite{BaakeBaake1} and is dealt with somewhat informally here. In continuous time,  simultaneous events are automatically excluded, so  the links are effectively independent and the fragmentation process simplifies noticeably. In fact, $(\wF{t}{})_{t \geqslant 0}^{}$ is then  defined as the following  continuous-time Markov chain with values in $\Pot{L}$ and  initial state $\wF{0}{} = \varnothing$: Conditional on $\wF{t}{}=G$,  $G\subseteq L$,  a transition to $G\cup\{\alpha\}$ occurs at rate $\rc{\alpha}>0$, for every $\alpha\in L \setminus G$.

If we denote by $\t{\alpha}$ the waiting time for link $\alpha$ to be removed and by $\t{K}:=\min\{\t{\alpha}:\alpha\in K\}$ the time at which the first link in $K\subseteq L$ is removed, then the $\t{\alpha}$ follow independent exponential distributions with parameters $\rc{\alpha}$. The explicit expression for the probability of the fragmentation process to be in state $G$ at time $t>0$ is therefore immediate:
\begin{align*}
\P\big(\wF{t}{}=G\big) & = \P\big(\max\{\t{\alpha}:\alpha\in G\}\leqslant t < \t{ L\setminus G} \big) =  \exp\Big (-\sum_{\alpha\in L\setminus G} \rc{\alpha}~ t \Big) \prod\limits_{\alpha\in G}\big(1-\exp(-\rc{\alpha}~ t)\big).
\end{align*}

For comparison with  discrete time, it is nonetheless interesting to additionally consider the tree probabilities in
 continuous time. As a matter of fact, Corollary~\ref{corollary:prob_tree} also holds for continuous time,  since it is a general
  statement in terms of waiting times that is not tied to the specific law in discrete time. Evaluating Corollary~\ref{corollary:prob_tree} for  independent  exponential  $\t{\alpha}$ and using the fact that the minimum of a collection of
   independent exponential waiting times is independent of the order in which the events appear gives the analogue of
    Theorem~\ref{prop_discrete}:

\begin{corol} \label{prop_continuous} For a given fragmentation tree $\Ts=(\gamma,G,E,L)$ and a fixed $t\geqslant 0$, one has $\P\big(\Ftree\big)=\exp(-\sum\limits_{\alpha\in L} \rc{\alpha}\, t)$ for $G=\varnothing$ and, for every $\varnothing \neq G\subseteq L$,
\begin{equation*}  \label{eq:prop_tree_continuous}
\P\big(\Ftree\big)=\!\sum\limits_{H\subseteq E} (-1)_{}^{|H|}\Big(1-\exp \Big(-\sum\limits_{\alpha\in \G{\gamma}{H}} \rc{\alpha}\,t\Big)\Big)\exp\Big(-\sum\limits_{\beta\in L\backslash \G{\gamma}{H}} \rc{\beta}\,t \Big) \prod\limits_{\alpha\in G} \frac{\rc{\alpha}}{\sum\limits_{\nu\in \G{\alpha}{H}} \rc{\nu}}.
\end{equation*}
\end{corol}

\subsection*{Acknowledgements}
It is our pleasure to thank Michael Baake for valuable discussions and Fernando Cordero for his help to improve the manuscript. The authors gratefully acknowledge
the support from the Priority Programme \emph{Probabilistic Structures in
Evolution (SPP 1590)}, which is funded by Deutsche Forschungsgemeinschaft
(German Research Foundation, DFG).

\bibliographystyle{plainnat} 
\renewcommand{\bibfont}{\small}
\setlength{\bibsep}{1pt}

\end{document}